%% file: main.tex
\documentclass[11pt,a4paper]{article}

\usepackage[T1]{fontenc}
\usepackage[utf8]{inputenc}
\usepackage{lmodern}
\usepackage[margin=1in]{geometry}
\usepackage{amsmath,amssymb,amsthm}
\usepackage{graphicx}
\usepackage{flafter}
\usepackage{subcaption}
\usepackage{booktabs,multirow}
\usepackage{algorithm}
\usepackage{algorithmic}
\usepackage[numbers,sort&compress]{natbib}
\usepackage{microtype}
\usepackage[hidelinks]{hyperref}

\hypersetup{
  pdftitle={High Performance Computing of SDRE models with Discrete Kalman Filtering for Robust H-infinity-Controls},
  pdfauthor={Yunfeng Cai, Tiexiang Li, Wen-Wei Lin, Junxin Zhang},
  pdfkeywords={SDRE, H-infinity control, Kalman filtering, SDA, Newton-Kleinman}
}
\allowdisplaybreaks
\numberwithin{equation}{section}
\newtheorem{Theorem}{Theorem}[section]

\newtheorem{Proposition}[Theorem]{Proposition}

\theoremstyle{remark}
\newtheorem{Remark}{Remark}[section]

\theoremstyle{definition}
\newtheorem{Definition}{Definition}[section]

\def\bb{\mathbf{b}}
\def\be{\mathbf{e}}
\def\bd{\mathbf{d}}

\def\bu{\mathbf{u}}

\def\bw{\mathbf{w}}
\def\bx{\mathbf{x}}
\def\by{\mathbf{y}}
\def\bs{\mathbf{s}}
\def\bz{\mathbf{z}}

\input{generated/paper_numbers.tex}

\input{generated/paper_feedforward_numbers.tex}

\title{High Performance Computing of SDRE models with Discrete Kalman Filtering for Robust \texorpdfstring{$H_\infty$}{H-infinity}-Controls}
\author{%
\normalsize Yunfeng Cai$^{1}$\quad
Tiexiang Li$^{2,3}$\thanks{Corresponding author: \href{mailto:txli@seu.edu.cn}{txli@seu.edu.cn}.}\quad
Wen-Wei Lin$^{3}$\quad Junxin Zhang$^{3,4}$\\[1em]
\parbox{0.96\textwidth}{\centering\small
$^{1}$Beijing Institute of Mathematical Sciences and Applications (BIMSA),\\
Beijing 101408, China\\[0.4em]
$^{2}$School of Mathematics and Shing-Tung Yau Center,\\
Southeast University, Nanjing 211189, China\\[0.4em]
$^{3}$Shanghai Institute for Mathematics and Interdisciplinary Sciences (SIMIS),\\
Shanghai 200433, China\\[0.4em]
$^{4}$Research Institute of Intelligent Complex Systems,\\
Fudan University, Shanghai 200433, China\\[0.7em]
\href{mailto:caiyunfeng@bimsa.cn}{\texttt{caiyunfeng@bimsa.cn}}\quad
\href{mailto:txli@seu.edu.cn}{\texttt{txli@seu.edu.cn}}\\
\href{mailto:wwlin@outlook.com}{\texttt{wwlin@outlook.com}}\quad
\href{mailto:jxzhang@simis.cn}{\texttt{jxzhang@simis.cn}}
}}
\date{}

\begin{document}
\maketitle

\begin{abstract}
State-dependent Riccati equation (SDRE) control requires repeated online solution of continuous-time algebraic Riccati equations (CAREs). A projected-channel $H_\infty$ formulation preserves the physical actuator and sensor dimensions of under-actuated, partially observed systems. We derive a computable sufficient attenuation bound for positive-semidefinite stabilizing Riccati solutions and construct an update that enforces their spectral-radius coupling condition. A discrete Kalman recursion supplies the scheduling state estimates. The two CAREs arising in controller synthesis are solved by a structure-preserving doubling algorithm (SDA) or a warm-started Newton--Kleinman iteration equipped with a doubling Lyapunov solver. Numerical studies use a twelve-state F-16 model with complete aerodynamic forces, moments, and trim, together with quadrotor spiral tracking. With instantaneous command updates, all three CARE backends give essentially the same closed-loop response. In the tracking simulations, measured computation times determine when new control inputs are applied, and the previous inputs are held during computation. SDA and Newton reduce airspeed and pitch errors in all six paired aircraft tracking runs. In the faster quadrotor task with limited computational resources, SDA and Newton complete all paired runs, while runs using MATLAB \texttt{icare} terminate early.
\end{abstract}

{\small\noindent\textbf{Keywords:} state-dependent Riccati equation; $H_\infty$ control; Kalman filtering; structure-preserving doubling; Newton--Kleinman method; under-actuated systems.\par}

\section{Introduction}

Nonlinear feedback methods for aerial vehicles address state-dependent dynamics and external disturbances in trajectory tracking and attitude stabilization.

Among these methods, the state-dependent Riccati equation (SDRE) approach provides a systematic framework for suboptimal nonlinear feedback design and is often viewed as a nonlinear counterpart of linear quadratic regulator (LQR) control~\cite{clou:1997,Cime:2010,Cime:2012}. SDRE represents the nonlinear system in a state-dependent coefficient (SDC) form and constructs a feedback law from a Riccati equation associated with the instantaneous frozen state~\cite{livl:2015,lilc:2018}. In this way, the state dependence of the nonlinear dynamics is retained while the familiar Riccati-based synthesis machinery of linear control remains applicable. This formulation requires repeated Riccati solves as the state evolves. Consequently, the online solution of a sequence of closely related continuous-time algebraic Riccati equations (CAREs) becomes a major computational bottleneck in SDRE-based control~\cite{livl:2015,lilc:2018}, and efficient structure-preserving numerical methods are important for reducing the resulting computational cost~\cite{lili:2021,hklw:2024,hukl:2026}.

Robustness against external disturbances and model uncertainties can be incorporated through an $H_\infty$ performance criterion. In the classical measurement-output-feedback setting, the central $H_\infty$-controller is characterized by two Riccati equations together with a spectral-radius condition~\cite{gd:1988,dgkf:1989}; invariant-subspace characterizations are also available~\cite{lxy:2001}. Nonlinear $H_\infty$ synthesis based on SDRE ideas has subsequently been investigated in various forms~\cite{fgim:2015}. For under-actuated and partially observed systems, an additional issue is to formulate the disturbance and performance channels in a way that is compatible with the physical actuator and sensor dimensions while retaining a realizable output-feedback controller.

Implementing state-dependent feedback requires state information, whereas the available measurements may be noisy and incomplete. Kalman filtering provides a natural mechanism for estimating the system state from such measurements~\cite{simo:2006}, and SDC-based filtering approaches use state-dependent coefficient representations of nonlinear systems~\cite{beit:2012}. When a continuous-time state-dependent estimator recomputes a steady-state Kalman gain at each frozen state, an additional CARE must be solved for state estimation. For a fixed attenuation level, this formulation involves three CAREs: two for the $H_\infty$ controller and one for the estimator. In this work, a discrete Kalman recursion updates the covariance and gain through prediction and correction, leaving repeated CARE computation to the two controller equations.

We therefore focus on the repeated solution of the two $H_\infty$ CAREs arising in the controller synthesis. For CARE computation, Newton--Kleinman iterations provide a classical iterative approach~\cite{klei:1968,feit:2009}, whereas Schur and generalized-eigenvalue methods provide standard direct alternatives~\cite{laub:1979,arla:1984}. Doubling methods range from early continuous-time constructions~\cite{kimu:1989} to structure-preserving algorithms for CAREs~\cite{chfl:2005}, together with subsequent analyses of convergence~\cite{lixu:2006}, weakly stabilizing solutions~\cite{huli:2009}, and shift-parameter selection~\cite{hull:2017}; broader accounts and comparisons are given in~\cite{hull:2018}. Along an SDRE trajectory, successive frozen states generate a sequence of nearby CAREs. This repeated-solve structure makes it natural to reuse information from preceding states and motivates numerical strategies specifically designed for efficient online updates.

Riccati-based control methods have been used extensively in aerial-vehicle applications. Representative studies include aircraft stabilization~\cite{cpws:2022}, computationally efficient SDRE control of a helicopter benchmark~\cite{lili:2021}, and SDRE or $H_\infty$ designs for quadrotors~\cite{neao:2022,ChHu:2022}; a tutorial treatment of multirotor modeling, estimation, and control is provided in~\cite{makc:2012}. Practical precedents for online and under-actuated SDRE implementation include real-time control of an under-actuated robot~\cite{eraa:2001}, autonomous-helicopter flight control~\cite{bowa:2007}, and recent onboard quadrotor flight tests~\cite{neol:2025}.

Against this background, this work develops an efficient computational framework for SDRE models with discrete Kalman filtering and measurement-output-feedback $H_\infty$ control, with applications to aircraft and quadrotor systems. The framework has three main components. First, for under-actuated and partially observed systems, the disturbance and performance channels are projected onto the column space of the actuator matrix and the row space of the measurement matrix, respectively, while the physical actuator and sensor dimensions are retained. For the resulting projected formulation, we establish an explicit sufficient attenuation bound under which the associated CAREs admit stabilizing positive-semidefinite solutions. Using classical Riccati monotonicity~\cite{Hewer:1993,lich:1993}, we obtain an attenuation update that enforces condition (C3). Second, the repeated controller CAREs are solved by a structure-preserving doubling algorithm (SDA) and by a warm-started Newton--Kleinman iteration equipped with a doubling Lyapunov solver. The latter reuses the preceding accepted Riccati solutions along the frozen-state trajectory and employs a stabilizing safeguard when necessary. Third, a discrete Kalman recursion replaces the estimator CARE by recursive matrix operations, thereby removing one Riccati solve from every frozen-state update.

The key contributions of this paper are summarized as follows.

\begin{itemize}

\item \textbf{A realizable projected-channel formulation and frozen-state feasibility result}. We formulate the measurement-output-feedback $H_\infty$ SDRE problem by projecting the disturbance and performance channels onto the column space of the actuator matrix and the row space of the measurement matrix, respectively, while preserving the physical actuator and sensor dimensions. For the projected CAREs, we derive an explicit sufficient attenuation bound ensuring the existence of stabilizing positive-semidefinite solutions. We use classical Riccati monotonicity to obtain an attenuation update with a prescribed margin for condition (C3).

\item \textbf{Efficient repeated CARE solution for SDRE--$H_\infty$ control}. We employ SDA~\cite{lixu:2006,hull:2018} and a Newton--Kleinman-type method~\cite{hklw:2024} for the two controller CAREs arising at each frozen state. The Newton implementation exploits the sequential nature of SDRE synthesis by warm-starting from the preceding accepted Riccati solutions and incorporates a stabilizing safeguard. Its repeated-solve cost is compared with SDA and the MATLAB solver \texttt{icare}.

\item \textbf{Discrete Kalman recursion for computational efficiency}. We integrate discrete Kalman prediction and correction with SDRE controller synthesis. The recursion updates the state estimate, covariance, and gain, and the posterior estimate determines the state-dependent coefficients used in controller synthesis. Recursive covariance propagation avoids an additional estimator CARE at each frozen-state update, leaving the two $H_\infty$ controller CAREs as the repeated algebraic Riccati solves.

\item \textbf{Numerical validation on aerial-vehicle models}. The framework is tested on a twelve-state F-16 model with complete aerodynamic forces, moments, and straight-level trim~\cite{isrlab:F16Model}, and on nominal and faster quadrotor spiral tracking~\cite{ChHu:2022}. Instantaneous updates yield essentially solver-independent accuracy. Measured computation costs are further incorporated into the command-update schedule of the tracking experiments.

\end{itemize}

Throughout this paper, $I_k$ denotes the $k\times k$ identity matrix,
${\rm Im}=i\mathbb{R}$ denotes the imaginary axis, and $\sigma(M)$ denotes
the spectrum of a square matrix $M$.

The remainder of the paper is organized as follows. Section~2 reviews the frozen-state $H_\infty$ output-feedback conditions. Section~3 develops the projected-channel formulation and the numerical methods for the repeated controller CAREs. Section~4 introduces the discrete Kalman state-estimation recursion. Section~5 presents the aircraft and quadrotor models, Section~6 reports the numerical experiments, and Section~7 concludes the paper.

\section{SDRE for \texorpdfstring{$H_\infty$}{H-infinity}-control systems}
We consider nonlinear measurement-feedback $H_\infty$ control under process
disturbances and sensor noise using state-dependent Riccati equations
(SDREs) \cite{simo:2006,fgim:2015}. In common-input notation, the generalized
plant $\mathsf{G}$ is represented as
\begin{subequations}\label{eq2.1}
\begin{align}
 \dot{\bx}&=A\bx+B_1\bw+B_2\bu,\quad\bx(0)=\bx_0,\label{eq2.1a}\\
 \bz&=C_1\bx+D_{12}\bu,\label{eq2.1b}\\
 \by&=C_2\bx+D_{21}\bw.\label{eq2.1c}
\end{align}
\end{subequations}
Here $\bx\in\mathbb{R}^n$, $\bw\in\mathbb{R}^{m_1}$,
$\bu\in\mathbb{R}^{m_2}$, $\bz\in\mathbb{R}^{p_1}$ and
$\by\in\mathbb{R}^{p_2}$ are the state, exogenous input, control,
performance output and measurement output, respectively, and
$\bx_0\in\mathbb{R}^n$ is the initial state.
The matrices $A\equiv A(\bx)\in\mathbb{R}^{n\times n}$,
$B_i\equiv B_i(\bx)\in\mathbb{R}^{n\times m_i}$,
$C_i\equiv C_i(\bx)\in\mathbb{R}^{p_i\times n}$,
$D_{12}\equiv D_{12}(\bx)\in\mathbb{R}^{p_1\times m_2}$ and
$D_{21}\equiv D_{21}(\bx)\in\mathbb{R}^{p_2\times m_1}$, with $i=1,2$,
may depend on $\bx$.

To distinguish process disturbances from sensor noise, write
\begin{align*}
 \bw&=[\bw_d^{\top},\bw_n^{\top}]^{\top},\qquad m_1=m_d+m_n,\\
 B_1&=[B_{1d},0_{n\times m_n}],\\
 D_{21}&=[0_{p_2\times m_d},D_{21,n}],
\end{align*}
where $\bw_d\in\mathbb{R}^{m_d}$ is the process disturbance in the chosen
coordinates, $\bw_n\in\mathbb{R}^{m_n}$ is sensor noise, and
$B_{1d}\in\mathbb{R}^{n\times m_d}$ and
$D_{21,n}\in\mathbb{R}^{p_2\times m_n}$ are the corresponding channel
blocks. Thus $B_1D_{21}^{\top}=0$.

For synthesis, we fix a frozen state and evaluate all coefficient matrices
at that state. The transfer matrices and linear-system results below refer
to this frozen realization, with the state dependence suppressed in the notation.
In the tracking applications, the reference data are fixed at the synthesis
instant together with the scheduling state.
In the model in \eqref{eq2.1}, the direct feedthrough matrices from $\bw$ to $\bz$
and from $\bu$ to $\by$ satisfy $D_{11}=0$ and $D_{22}=0$.
We impose the following assumptions on the resulting frozen matrices:
\begin{itemize}
\item[(A1)] $(A,B_2)$ is  stabilizable and $(C_1,A)$ is  detectable;
\item[(A2)] $(A,B_1)$ is  stabilizable and $(C_2,A)$ is  detectable;
\item[(A3)] The feedthrough matrices satisfy
\begin{align*}
 D_{12}^{\top}[C_1,D_{12}]&=[0_{m_2\times n},I_{m_2}],\\
 D_{21}[B_1^{\top},D_{21}^{\top}]&=[0_{p_2\times n},I_{p_2}].
\end{align*}
\end{itemize}

When (A3) is not satisfied, a general treatment is given
in~\cite{gd:1988}. For convenience, we describe here the normalization
for the case $D_{12}^{\top}C_1=0$ and $B_1D_{21}^{\top}=0$.
When $D_{12}$ has full column rank and
$D_{21}$ has full row rank, define
\begin{align*}
 R_u&=D_{12}^{\top}D_{12}\in\mathbb{R}^{m_2\times m_2},&
 R_y&=D_{21}D_{21}^{\top}\in\mathbb{R}^{p_2\times p_2}.
\end{align*}
Both matrices are positive definite. The input scaling and measurement
whitening
$\bu_{\mathrm{N}}=R_u^{1/2}\bu$ and
$\by_{\mathrm{N}}=R_y^{-1/2}\by$ give
\begin{align*}
 B_{2,\mathrm{N}}&=B_2R_u^{-1/2},&
 D_{12,\mathrm{N}}&=D_{12}R_u^{-1/2},\\*
 C_{2,\mathrm{N}}&=R_y^{-1/2}C_2,&
 D_{21,\mathrm{N}}&=R_y^{-1/2}D_{21}.
\end{align*}
These transformations preserve the zero cross terms and give
$D_{12,\mathrm{N}}^{\top}D_{12,\mathrm{N}}=I_{m_2}$ and
$D_{21,\mathrm{N}}D_{21,\mathrm{N}}^{\top}=I_{p_2}$, so the transformed
plant satisfies (A3). The normalized controller is implemented through
$\by_{\mathrm{N}}=R_y^{-1/2}\by$ and
$\bu=R_u^{-1/2}\bu_{\mathrm{N}}$, preserving the closed-loop map
from $\bw$ to $\bz$. Below, we retain the original symbols for the
matrices satisfying (A1)--(A3), after normalization when needed.

We consider an $n$-state measurement-feedback controller $\mathsf{K}$
interconnected with the plant $\mathsf{G}$ as in Figure~\ref{fig1}(a).
Its state $\xi\in\mathbb{R}^n$ satisfies
\begin{align*}
    \left\{\begin{array}{l}
         \dot{\xi}=\widehat{A}_0 \xi+\widehat{B}_0\by, \\
         \bu=\widehat{C}_0\xi, 
    \end{array}\right.
\end{align*}
in which $\widehat{A}_0\in\mathbb{R}^{n\times n}$,
$\widehat{B}_0\in\mathbb{R}^{n\times p_2}$ and
$\widehat{C}_0\in\mathbb{R}^{m_2\times n}$ are the controller matrices to be determined.
We denote the closed-loop system obtained by interconnecting $\mathsf{G}$
and $\mathsf{K}$ by $\widehat{\mathsf{G}}$, with input $\bw\in\mathbb{R}^{m_1}$
and output $\bz\in\mathbb{R}^{p_1}$.
With $\bs=[\bx^{\top},\xi^{\top}]^{\top}\in\mathbb{R}^{2n}$,
its realization in Figure~\ref{fig1}(b) is
\begin{subequations}\label{eq2.3}
\begin{align}
 \dot{\bs}&=\mathcal{A}\bs+\mathcal{B}\bw,
 &\bz&=\mathcal{C}\bs,\label{eq2.3a}\\
 \mathcal{A}&=\begin{bmatrix}
 A&B_2\widehat{C}_0\\
 \widehat{B}_0C_2&\widehat{A}_0
 \end{bmatrix},
 &\mathcal{B}&=\begin{bmatrix}B_1\\\widehat{B}_0D_{21}\end{bmatrix},
 \label{eq2.3b}\\
 \mathcal{C}&=\begin{bmatrix}C_1&D_{12}\widehat{C}_0\end{bmatrix}.
 \label{eq2.3c}
\end{align}
\end{subequations}
Here $\mathcal{A}\in\mathbb{R}^{2n\times2n}$,
$\mathcal{B}\in\mathbb{R}^{2n\times m_1}$ and
$\mathcal{C}\in\mathbb{R}^{p_1\times2n}$.

\begin{figure}[tbp]
    \centering
    \begin{subfigure}[t]{0.48\textwidth}
        \centering
        \includegraphics[width=\linewidth]{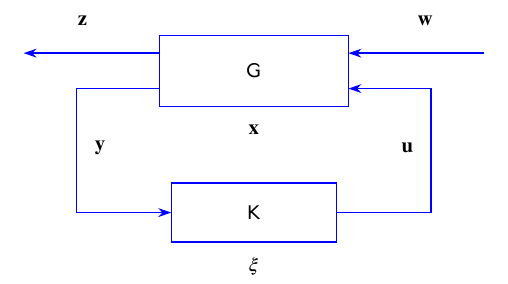}
        \caption{Plant $\mathsf{G}$ interconnected with controller $\mathsf{K}$.}
    \end{subfigure}\hfill
    \begin{subfigure}[t]{0.48\textwidth}
        \centering
        \includegraphics[width=\linewidth]{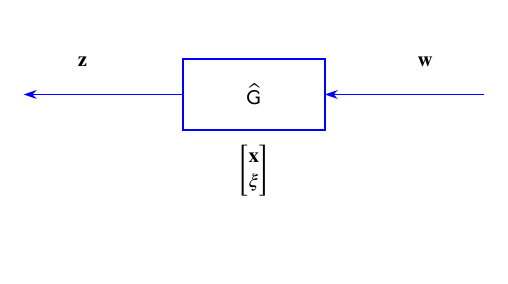}
        \caption{Equivalent closed-loop realization $\widehat{\mathsf{G}}$.}
    \end{subfigure}
    \caption{Measurement-output-feedback $H_{\infty}$ interconnection and its closed-loop representation.}
    \label{fig1}
\end{figure}

At a fixed frozen state, the transfer matrix of $\widehat{\mathsf{G}}$
from $\bw$ to $\bz$ is
\begin{align*}
    T_{\bz\bw}(s)=\mathcal{C}(sI_{2n}-\mathcal{A})^{-1}\mathcal{B}
    \in\mathbb{C}^{p_1\times m_1},
\end{align*}
where $s\in\mathbb{C}$ lies outside $\sigma(\mathcal{A})$.
For a stable closed-loop realization, its $H_\infty$ norm is $\|\widehat{\mathsf{G}}\|_{\infty}\equiv \|T_{\bz\bw}\|_{\infty}:= \sup_{\omega\in \mathbb{R}}\|T_{\bz\bw}(i\omega)\|_2$. 

For a prescribed $\gamma>0$, the synthesis objective is to choose
$\mathsf{K}$ so that $\mathcal{A}$ is stable and
$\|T_{\bz\bw}\|_\infty<\gamma$. A controller that internally stabilizes
the frozen plant is called admissible.
We recall the definitions of Hamiltonian matrices and the Riccati operator,
together with the Lyapunov theorem and the strict bounded real lemma.

\begin{Definition}[Hamiltonian matrices and Riccati domain {\cite{dgkf:1989}}]\label{def:ric-domain}
Let $r$ be a positive integer. A matrix $H\in\mathbb{R}^{2r\times2r}$
is called Hamiltonian if
\begin{align*}
 H^{\top}J_r+J_rH=0,\qquad
 J_r=\begin{bmatrix}0&I_r\\-I_r&0\end{bmatrix}
 \in\mathbb{R}^{2r\times2r}.
\end{align*}
Equivalently, $H$ has the block form
\begin{align*}
 H=\begin{bmatrix}F&-R\\-Q&-F^{\top}\end{bmatrix},
\end{align*}
where $F,R,Q\in\mathbb{R}^{r\times r}$,
$R=R^{\top}$ and $Q=Q^{\top}$.
We write $H\in\operatorname{Dom}(\operatorname{Ric})$ if and only if
$H$ has no eigenvalues on the imaginary axis and its $r$-dimensional
stable invariant subspace has a basis
\begin{align*}
 U=\begin{bmatrix}U_1\\U_2\end{bmatrix}\in\mathbb{R}^{2r\times r},
 \qquad U_1,U_2\in\mathbb{R}^{r\times r},
\end{align*}
with $U_1$ nonsingular. In that case, define
\begin{align*}
 \operatorname{Ric}(H):=P=U_2U_1^{-1}\in\mathbb{R}^{r\times r}.
\end{align*}
The matrix $P$ is independent of the chosen basis and is the unique
symmetric stabilizing solution of
\begin{align*}
 F^{\top}P+PF-PRP+Q=0,
\end{align*}
meaning that $F-RP$ is stable.
\end{Definition}

\begin{Theorem}[Lyapunov theorem {\cite{simo:2006}}]\label{thm2.2}
Let $F\in\mathbb{R}^{r\times r}$ be stable and let
$S=S^{\top}\in\mathbb{R}^{r\times r}$. The Lyapunov equation
\begin{align*}
 F^{\top}P+PF+S=0
\end{align*}
has a unique solution $P\in\mathbb{R}^{r\times r}$, given by
\begin{align*}
 P=\int_0^{\infty}e^{F^{\top}t}S e^{Ft}\,dt=P^{\top}.
\end{align*}
In particular, $S\ge0$ implies $P\ge0$, and $S\le0$ implies $P\le0$.
\end{Theorem}

\begin{Theorem}[Strict bounded real lemma {\cite[Lemma~3.1]{lxy:2001}}]\label{thm2.3}
Let $r,m,p$ be positive integers, let $F\in\mathbb{R}^{r\times r}$,
$B\in\mathbb{R}^{r\times m}$ and $C\in\mathbb{R}^{p\times r}$,
and let $\gamma>0$. Define
\begin{align*}
 T(s)&=C(sI_r-F)^{-1}B\in\mathbb{C}^{p\times m},\\
 H&=\begin{bmatrix}F&\gamma^{-2}BB^{\top}\\-C^{\top}C&-F^{\top}\end{bmatrix}
 \in\mathbb{R}^{2r\times2r},
\end{align*}
where $s\in\mathbb{C}\setminus\sigma(F)$.
Then the following statements are equivalent:
\begin{itemize}
 \item[(i)] There is $P\in\mathbb{R}^{r\times r}$ with $P=P^{\top}\ge0$ such that
 \begin{align}\label{eq2.6}
 H\begin{bmatrix}I_r\\P\end{bmatrix}
 =\begin{bmatrix}I_r\\P\end{bmatrix}W,
 \end{align}
 where $W=F+\gamma^{-2}BB^{\top}P\in\mathbb{R}^{r\times r}$ is stable.
 \item[(ii)] $F$ is stable and $\|T\|_\infty<\gamma$.
\end{itemize}
\end{Theorem}

Applying Theorem~\ref{thm2.3} to $\widehat{\mathsf{G}}$ in \eqref{eq2.3}
with $r=2n$, $F=\mathcal{A}$, $B=\mathcal{B}$ and $C=\mathcal{C}$
gives $T=T_{\bz\bw}$ and the closed-loop Hamiltonian
$H=\mathcal{H}(\gamma)\in\mathbb{R}^{4n\times4n}$:
\begin{align}\label{eq2.4}
    \mathcal{H}(\gamma)=\left[\begin{array}{cc}
             \mathcal{A}&\frac{1}{\gamma^2}\mathcal{B}\mathcal{B}^{\top} \\
             -\mathcal{C}^{\top}\mathcal{C}&-\mathcal{A}^{\top}  
        \end{array}\right].
\end{align}

Theorem~\ref{thm2.3} assesses a given controller through its closed-loop
realization. We now turn to constructing an admissible controller with the
prescribed attenuation level. The normalized output-feedback result
\cite{gd:1988,dgkf:1989} gives existence conditions in terms of two
$n\times n$ Riccati solutions, denoted by $X$ and $Y$.

\begin{Theorem}\label{thm:DGKF}
At a fixed frozen state, suppose that Assumptions (A1)--(A3) hold, and let $\gamma>0$.
There exists an admissible controller $\mathsf{K}=(\widehat{A}_0,\widehat{B}_0,\widehat{C}_0)$ for \eqref{eq2.3} such that
\begin{subequations}\label{eq2.8}
\begin{align}
 &\|\widehat{\mathsf{G}}\|_{\infty}
   =\|T_{\bz\bw}\|_{\infty}< \gamma,\label{eq2.8a}\\
 &\begin{bmatrix}A&B_2\widehat{C}_0\\
 \widehat{B}_0C_2&\widehat{A}_0\end{bmatrix}
 \quad\text{is stable},\label{eq2.8b}
\end{align}
\end{subequations}
if and only if 
\begin{itemize}
    \item[(C1)] $H_{\infty}(\gamma)\in \operatorname{Dom}(\operatorname{Ric})$ and $X\equiv X(\gamma)={\rm Ric}(H_{\infty}(\gamma))\in\mathbb{R}^{n\times n}$, where $H_{\infty}(\gamma)\in\mathbb{R}^{2n\times2n}$ is given by
    \begin{subequations}\label{eq2.9}
    \begin{align}\label{eq2.9a}
        &H_{\infty}(\gamma)=\left[\begin{array}{cc}
             A&-G_{\gamma}\\
             -C_1^{\top}C_1&-A^{\top}  
        \end{array}\right],\\
        &H_{\infty}(\gamma)\left[\begin{array}{c}
             I_n \\
             X 
        \end{array}\right]=\left[\begin{array}{c}
             I_n \\
             X 
        \end{array}\right]W_X
    \end{align}
    with $G_{\gamma}=B_2B_2^{\top}-{\gamma^{-2}}B_1B_1^{\top}\in\mathbb{R}^{n\times n}$, $X=X^{\top}\ge 0$, and $W_X=A-G_{\gamma}X\in\mathbb{R}^{n\times n}$ stable. Furthermore, $X$ satisfies the CARE
    \begin{equation}\label{eq2.9b}
        A^{\top}X+XA-XG_{\gamma}X+C_1^{\top}C_1=0.
    \end{equation}
    \end{subequations}
    \item[(C2)] $J_{\infty}(\gamma)\in \operatorname{Dom}(\operatorname{Ric})$ and $Y\equiv Y(\gamma)={\rm Ric}(J_{\infty}(\gamma))\in\mathbb{R}^{n\times n}$, where $J_{\infty}(\gamma)\in\mathbb{R}^{2n\times2n}$ is given by
    \begin{subequations}\label{eq2.10}
    \begin{align}\label{eq2.10a}
        &J_{\infty}(\gamma)\equiv \left[\begin{array}{cc}
             A^{\top}&-H_{\gamma}\\
             -B_1B_1^{\top}&-A  
        \end{array}\right],\\
        &J_{\infty}(\gamma)\left[\begin{array}{c}
             I_n \\
             Y 
        \end{array}\right]=\left[\begin{array}{c}
             I_n \\
             Y 
        \end{array}\right]W_Y
    \end{align}
    with $H_\gamma=C_2^{\top}C_2-\gamma^{-2}C_1^{\top}C_1\in\mathbb{R}^{n\times n}$, $Y=Y^{\top}\ge 0$ and $W_Y=A^{\top}-H_{\gamma}Y\in\mathbb{R}^{n\times n}$ stable. Furthermore, $Y$ satisfies the CARE
    \begin{align}\label{eq2.10b}
        AY+YA^{\top}-YH_{\gamma}Y+B_1B_1^{\top}=0.
    \end{align}
    \end{subequations}
    \item[(C3)] The spectral radius of $XY$ satisfies
    \begin{equation}\label{eq2.11}
        \rho(XY)<\gamma^2.
    \end{equation}
\end{itemize}
\end{Theorem}
The classical stable-invariant-subspace construction \cite{dgkf:1989,lxy:2001} applied to \eqref{eq2.9}--\eqref{eq2.10} gives the central $H_\infty$-controller $\mathsf{K}=(\widehat{A}_0,\widehat{B}_0,\widehat{C}_0)$ for measurement feedback \eqref{eq2.3} as
\begin{subequations}\label{eq2.12}
    \begin{align}
    \widehat{A}_0&=W_X-\widehat{B}_0C_2,\label{eq2.12a}\\
    \widehat{B}_0&=(I_n-\frac{1}{\gamma^2}YX)^{-1}YC_2^{\top},\label{eq2.12b}\\
    \widehat{C}_0&=-B_2^{\top}X,\label{eq2.12c}
\end{align}
\end{subequations}
where $X$, $Y$ and $W_X$ are defined by \eqref{eq2.9} and \eqref{eq2.10}.
Under Assumptions (A1)--(A3) and conditions (C1)--(C3), \eqref{eq2.12}
gives an admissible central controller for the frozen-state interconnection;
see \citet[Theorem~4.1, Eq.~(4.4), and its proof]{lxy:2001}.
\section{Efficient computation of the \texorpdfstring{$H_\infty$ controller $\mathsf{K}$}{H-infinity controller K}}
For the linear realization at each fixed state, this section establishes CARE feasibility and develops SDA and Newton-type iterations for repeatedly solving the two controller CAREs and checking conditions (C1)--(C3) in \eqref{eq2.9}--\eqref{eq2.11}.

\begin{Theorem}\label{thm3.2}
At a fixed frozen state, suppose that Assumptions (A1)--(A2) hold.
Replacing $\gamma^{-2}$ by $\alpha$ in (C1) and (C2) of
\eqref{eq2.9} and \eqref{eq2.10}, respectively, there exists an
$\alpha^*>0$ such that both conditions hold for every
$\alpha\in[0,\alpha^*]$.
\end{Theorem}
\begin{proof}
We first prove (C1). Consider the equations
\begin{subequations}\label{eq3.10}
\begin{align}
&H(\alpha)\equiv
\begin{bmatrix}
A&\alpha B_1B_1^{\top}-B_2B_2^{\top}\\
-C_1^{\top}C_1&-A^{\top}
\end{bmatrix},\\
&H(\alpha)\begin{bmatrix}I\\X(\alpha)\end{bmatrix}
=\begin{bmatrix}I\\X(\alpha)\end{bmatrix}W(\alpha).
\end{align}
\end{subequations}
From the detectability of $(C_1,A)$ and the stabilizability of $(A,B_2)$, it follows from \cite{valo:1984} that $H(0)\in \operatorname{Dom}(\operatorname{Ric})$ and $X(0)^{\top}=X(0)={\rm Ric}(H(0))\ge 0$. Since $H(0)$ has no imaginary-axis eigenvalues, its stable invariant
subspace varies continuously for sufficiently small $\alpha$.
The upper block of a continuous basis is invertible at $\alpha=0$
and remains invertible nearby. Hence there is an $\alpha_1>0$
such that the continuous solution curve
\begin{align}\label{eq3.11}
    \{(X(\alpha),W(\alpha))\ |\ \alpha\in [0,\alpha_1]\}
\end{align}
of \eqref{eq3.10} exists with $\sigma(H(\alpha))\bigcap {\rm Im}=\emptyset$ and $W(\alpha)$ stable. To establish the symmetry of $X(\alpha)$, proceed as follows. Write $G(\alpha)\equiv B_2B_2^{\top}-\alpha B_1B_1^{\top}$. Then from \eqref{eq3.10} we have
\begin{equation*}\begin{aligned}
X(\alpha)\,&(A-G(\alpha)X(\alpha))+(A-G(\alpha)X(\alpha))^{\top}X(\alpha)\\
&=-X(\alpha)^{\top}G(\alpha)X(\alpha)-C_1^{\top}C_1.
\end{aligned}\end{equation*}
Since $\widehat A=A-G(\alpha)X(\alpha)=W(\alpha)$ is stable,
Theorem~\ref{thm2.2}, with $F=\widehat A$ and
$S=X(\alpha)^{\top}G(\alpha)X(\alpha)+C_1^{\top}C_1=S^{\top}$,
gives
\begin{equation*}\begin{aligned}
X(\alpha)
&=\int_0^{\infty}e^{\widehat A^{\top}t}
\left(X(\alpha)^{\top}G(\alpha)X(\alpha)+C_1^{\top}C_1\right)e^{\widehat A t}\,dt\\
&=X(\alpha)^{\top}.
\end{aligned}\end{equation*}

Let $A(\alpha)=A-B_2B_2^{\top}X(\alpha)$. Since $A(0)=W(0)$ is stable and $X(\alpha)=X(\alpha)^{\top}$ is continuous, we can choose $0<\alpha_2\le\alpha_1$ such that $A(\alpha)$ is stable and $X(\alpha)$ is bounded for $\alpha\in[0,\alpha_2]$. From \eqref{eq3.10} we have
\begin{subequations}\label{eq3.15}
\begin{align}
&Q(\alpha)\equiv C_1^{\top}C_1
 +X(\alpha)B_2B_2^{\top}X(\alpha),\\
&M(\alpha)\equiv\begin{bmatrix}
A(\alpha)&\alpha B_1B_1^{\top}\\
-Q(\alpha)&-A(\alpha)^{\top}
\end{bmatrix},\\
&M(\alpha)\begin{bmatrix}I\\X(\alpha)\end{bmatrix}
=\begin{bmatrix}I\\X(\alpha)\end{bmatrix}W(\alpha),
\end{align}
\end{subequations}
for $\alpha\in[0,\alpha_2]$. 
We define
\begin{equation}\label{eq3.16}
C(\alpha)=\begin{bmatrix}C_1\\B_2^{\top}X(\alpha)\end{bmatrix},\quad\mathsf{G}_\alpha=C(\alpha)(sI-A(\alpha))^{-1}B_1
\end{equation}
for $s=i\omega\in {\rm Im}$. On the compact interval $[0,\alpha_2]$,
$A(\alpha)$ and $C(\alpha)$ are continuous, and $A(\alpha)$ is stable.
The resolvent is therefore uniformly bounded on bounded frequency
intervals and decays uniformly as $|\omega|\to\infty$.
Consequently, we can choose a constant
\begin{equation}\label{eq3.18}
    \gamma_1>\sup_{\alpha\in[0,\alpha_2]}
    \|\mathsf{G}_{\alpha}\|_{\infty}.
\end{equation}
Set $\alpha_X^*=\min\{\alpha_2,\gamma_1^{-2}\}>0$.
For $0<\alpha\le\alpha_X^*$, $A(\alpha)$ is stable and
$\|\mathsf{G}_{\alpha}\|_{\infty}<\gamma_1\le\alpha^{-1/2}$.
Apply Theorem~\ref{thm2.3} with $r=n$, $F=A(\alpha)$,
$B=B_1$, $C=C(\alpha)$ and $\gamma=\alpha^{-1/2}$, so that
$H=M(\alpha)$ in \eqref{eq3.15}. Since $X(\alpha)$ already defines a
stable graph subspace of $M(\alpha)$, uniqueness of the stabilizing
Riccati solution identifies it with $P$ in the theorem and gives
$X(\alpha)\ge0$. The case $\alpha=0$ was established at the start of the proof.
Thus (C1) holds on $[0,\alpha_X^*]$.
Applying the same argument with
$(A,B_1,B_2,C_1)$ replaced by
$(A^{\top},C_1^{\top},C_2^{\top},B_1^{\top})$ gives
(C2) on $[0,\alpha_Y^*]$ for some $\alpha_Y^*>0$, using (A2).
Taking $\alpha^*=\min\{\alpha_X^*,\alpha_Y^*\}$ completes the proof.
\end{proof}

Building on Theorem~\ref{thm3.2}, we construct projected channels that
yield an explicit sufficient attenuation threshold. Let $B_1^{\rm o}$
and $C_1^{\rm o}$ denote the normalized frozen channels $B_1$ and $C_1$
of Section~2 before projection. We project the columns of $B_1^{\rm o}$
onto $\operatorname{range}(B_2)$ and the transposed rows of $C_1^{\rm o}$
onto $\operatorname{range}(C_2^{\top})$. Define
\begin{subequations}\label{eq3.19proj}
\begin{align}
    &\widetilde B_1=P_BB_1^{\rm o}=B_2K_B,\,
    P_B=B_2B_2^\dagger,\,K_B=B_2^\dagger B_1^{\rm o},\\
    &\widetilde C_1=C_1^{\rm o}P_C=K_CC_2,\,
    P_C=C_2^\dagger C_2,\,K_C=C_1^{\rm o}C_2^\dagger.
\end{align}
\end{subequations}
Here $\dagger$ denotes the Moore--Penrose inverse, and $P_B$ and
$P_C$ are the orthogonal projectors onto $\operatorname{range}(B_2)$ and
$\operatorname{range}(C_2^{\top})$, respectively. In the remainder of this
section, the generalized plant, the two CAREs, and the controller recovery
in Section~2 use $(B_1,C_1)=(\widetilde B_1,\widetilde C_1)$.
Sufficient channel thresholds are
\begin{subequations}\label{eq3.19}
    \begin{align}
        \gamma_b&=\|K_B\|_2=\|B_2^\dagger B_1^{\rm o}\|_2,\label{eq3.19a}\\
        \gamma_c&=\|K_C\|_2=\|C_1^{\rm o}C_2^\dagger\|_2.\label{eq3.19b}
    \end{align}
\end{subequations} 
We then choose a strict margin above
\begin{equation}\label{eq3.20}
    \gamma>\gamma^*=\max\{\gamma_b,\gamma_c\}>0,
\end{equation}
which gives the factorizations
\begin{subequations}\label{eq3.21}
\begin{align}
G_\gamma
&=B_2B_2^{\top}-\gamma^{-2}\widetilde B_1\widetilde B_1^{\top}\notag\\
&=B_2(I-\gamma^{-2}K_BK_B^{\top})B_2^{\top}\ge0,\\
H_\gamma
&=C_2^{\top}C_2-\gamma^{-2}\widetilde C_1^{\top}\widetilde C_1\notag\\
&=C_2^{\top}(I-\gamma^{-2}K_C^{\top}K_C)C_2\ge0.
\end{align}
\end{subequations}
The inequalities are strict on ${\rm range}(B_2)$ and ${\rm range}(C_2^{\top})$, respectively, because $\gamma>\gamma^*$.

\begin{Theorem}\label{thm3.3}
    Suppose that $\gamma>\gamma^*>0$ is chosen as in \eqref{eq3.19}--\eqref{eq3.20}, $(A,B_2)$ and $(A,\widetilde B_1)$ are stabilizable, and $(\widetilde C_1,A)$ and $(C_2,A)$ are detectable. Then the projected Hamiltonians $H_{\infty}(\gamma)$ and $J_{\infty}(\gamma)$ belong to $\operatorname{Dom}(\operatorname{Ric})$ and have stabilizing solutions
    $X(\gamma)=X(\gamma)^{\top}\ge0$ and $Y(\gamma)=Y(\gamma)^{\top}\ge0$, respectively.
\end{Theorem}
\begin{proof}
Let $\delta_B=1-\gamma^{-2}\|K_B\|_2^2>0$ and
$\delta_C=1-\gamma^{-2}\|K_C\|_2^2>0$. Equation~\eqref{eq3.21} yields
\begin{align*}
G_\gamma\ge\delta_BB_2B_2^{\top},\qquad
H_\gamma\ge\delta_CC_2^{\top}C_2.
\end{align*}
Let $q\in\mathbb{C}^n\setminus\{0\}$ satisfy $q^*A=\lambda q^*$ with
${\rm Re}\lambda\ge0$, where $*$ denotes the conjugate transpose.
Stabilizability of $(A,B_2)$ implies $B_2^{\top}q\ne0$, and hence
\begin{align*}
q^*G_\gamma q\ge\delta_B\|B_2^{\top}q\|_2^2>0.
\end{align*}
Thus $(A,G_\gamma^{1/2})$ is stabilizable. Together with detectability of $(\widetilde C_1,A)$, standard positive-semidefinite CARE theory gives the stabilizing solution of
\begin{align*}
A^{\top}X+XA-XG_\gamma X+\widetilde C_1^{\top}\widetilde C_1=0.
\end{align*}

For the dual equation, $H_\gamma\ge\delta_CC_2^{\top}C_2$ and detectability of $(C_2,A)$ imply detectability of $(H_\gamma^{1/2},A)$; stabilizability of $(A,\widetilde B_1)$ then gives the stabilizing solution of
\begin{align*}
AY+YA^{\top}-YH_\gamma Y+\widetilde B_1\widetilde B_1^{\top}=0.
\end{align*}

Both solutions are symmetric positive semidefinite~\cite{Hewer:1993}.
\end{proof}

The following consequence of classical Riccati monotonicity gives
an attenuation update for enforcing (C3).

\begin{Proposition}[Attenuation update]\label{prop:c3-update}
At a fixed frozen state, suppose that the assumptions of
Theorem~\ref{thm3.3} hold. For $\gamma>\gamma^*$, define
$\rho(\gamma)=\rho(X(\gamma)Y(\gamma))$ and
\begin{align}\label{eq:c3-ratio}
\chi(\gamma)=\gamma^{-2}\rho(\gamma).
\end{align}
Then $\chi$ is nonincreasing and tends to zero as $\gamma\to\infty$.
Given $\gamma_l>\gamma^*$, let $\rho_l=\rho(\gamma_l)$.
For $\tau\in(0,1)$ and $\eta>0$, the update
\begin{align}\label{eq:c3-update}
\gamma_{l+1}=(1+\eta)\max\left\{\gamma_l,
\sqrt{\frac{\rho_l}{1-\tau}}\right\}
\end{align}
satisfies $\chi(\gamma_{l+1})<1-\tau$ after recomputing the two
CARE solutions, in exact arithmetic.
\end{Proposition}
\begin{proof}
Classical Riccati monotonicity gives
$0\le\rho(\gamma_2)\le\rho(\gamma_1)$ for
$\gamma_2\ge\gamma_1>\gamma^*$~\cite{Hewer:1993,lich:1993}.
Thus $\chi$ is nonincreasing and, for any fixed $\gamma_0>\gamma^*$,
$0\le\chi(\gamma)\le\rho(\gamma_0)/\gamma^2\to0$ as $\gamma\to\infty$.
Moreover, \eqref{eq:c3-update} gives
$\gamma_{l+1}>\gamma_l$ and
$\gamma_{l+1}^2>\rho_l/(1-\tau)$, while monotonicity gives
$\rho(\gamma_{l+1})\le\rho_l$. The claimed margin follows.
\end{proof}

We first apply SDA to the projected Hamiltonian for $X(\gamma)$:
\begin{align*}
 H=H_\infty(\gamma)=\begin{bmatrix}A&-G\\-Q&-A^{\top}\end{bmatrix}
 \in\mathbb{R}^{2n\times2n},
\end{align*}
where $G=G_\gamma\ge0$ and $Q=\widetilde C_1^{\top}\widetilde C_1\ge0$.
Applying the Cayley transformation
$(H-\mu I_{2n})^{-1}(H+\mu I_{2n})$, with $\mu>0$ chosen so that
$H-\mu I_{2n}$ and $A-\mu I_n$ are nonsingular, gives the
corresponding symplectic pair \cite{lixu:2006,hull:2018} as
    \begin{subequations}\label{eq3.22}
    \begin{align}
        \mathcal{M}=\left[\begin{array}{cc}
            A_0 & 0 \\
            - H_0 & I
        \end{array}\right],\ \mathcal{L}&=\left[\begin{array}{cc}
            I & G_0 \\
            0 & A_0^{\top}
        \end{array}\right],\label{eq3.22a}\\
        \mathcal{M}J\mathcal{M}^{\top}
        &=\mathcal{L}J\mathcal{L}^{\top}.\label{eq3.22b}
    \end{align}
    \end{subequations}
    Here $J=\begin{bmatrix}0&I_n\\-I_n&0\end{bmatrix}\in\mathbb{R}^{2n\times2n}$,
    $A_\mu=A-\mu I_n$, and
    \begin{align*}
    A_0&=I_n+2\mu(A_\mu+GA_\mu^{-\top}Q)^{-1},\\*
    G_0&=2\mu A_\mu^{-1}G
    (A_\mu^{\top}+QA_\mu^{-1}G)^{-1},\\*
    H_0&=2\mu(A_\mu^{\top}+QA_\mu^{-1}G)^{-1}
    QA_\mu^{-1}.
    \end{align*}
    For $G,Q\ge0$, nonsingularity of $A_\mu$ also ensures that
    $A_\mu^{\top}+QA_\mu^{-1}G$ is nonsingular.
    Subject to these nonsingularity conditions, a practical shift can be
    selected using the condition-number-based Fibonacci search of
    \citet[Section~3.1]{chfl:2005}. Optimal shifts based on spectral
    enclosure regions are studied by \citet{hull:2017}.
    The pencil $(\mathcal{M},\mathcal{L})$ and $H$ share the stable invariant subspace ${\rm span}([I,X]^{\top})$. Since $\sigma (H)\bigcap {\rm Im}=\emptyset$, the pencil has no unimodular eigenvalues. Starting from $(A_0,G_0,H_0)$, with $G_0=G_0^{\top}\ge0$ and $H_0=H_0^{\top}\ge0$, we perform the SDA algorithm \cite{hull:2018,lixu:2006} with quadratic convergence as
    \begin{align*}
        A_{j+1}&=A_{j}(I+G_{j}H_{j})^{-1}A_{j}\rightarrow 0,\\
        H_{j+1}&=H_{j}+A_{j}^{\top}(I+H_{j}G_{j})^{-1}H_{j}A_{j}\rightarrow X\ge 0,\\
        G_{j+1}&=G_{j}+A_{j}G_{j}(I+H_{j}G_{j})^{-1}A_{j}^{\top}\rightarrow Z.
    \end{align*}
    Here $Z$ is the stabilizing Riccati solution associated with the dual
    Hamiltonian $\left[\begin{array}{cc}
        A^{\top} &-Q  \\
        -G & -A
    \end{array}\right]$, whose stable invariant subspace is
    ${\rm span}([I,Z]^{\top})$.

Thus $X=X(\gamma)={\rm Ric}(H_{\infty}(\gamma))\ge0$.
To obtain $Y(\gamma)$, apply the same iteration with $(A,G,Q)$ replaced by
$(A^{\top},H_\gamma,\widetilde B_1\widetilde B_1^{\top})$.

Alternatively, consider the CARE $F^{\top}U+UF-UGU+Q=0$,
where $F,G,Q\in\mathbb{R}^{n\times n}$ and $G,Q$ are symmetric.
For $U=X$, take $(F,G,Q)=(A,G_\gamma,\widetilde C_1^{\top}\widetilde C_1)$;
for $U=Y$, take $(F,G,Q)=(A^{\top},H_\gamma,\widetilde B_1\widetilde B_1^{\top})$.
Starting from a symmetric $U_0$ with $F-GU_0$ stable, the
Newton--Kleinman iteration~\cite{klei:1968} solves
\begin{align}\label{eq:newton-lyapunov}
 L_j^{\top}U_{j+1}+U_{j+1}L_j+Q_j=0,
\end{align}
where $L_j=F-GU_j$ and $Q_j=Q+U_jGU_j$.
For stable $L_j$, this Lyapunov equation has a unique solution by
Theorem~\ref{thm2.2}. At each Newton step, choose $\mu>0$ and initialize
\begin{align*}
 E_0&=(L_j-\mu I_n)^{-1}(L_j+\mu I_n),\\
 P_0&=2\mu(L_j-\mu I_n)^{-\top}Q_j(L_j-\mu I_n)^{-1}.
\end{align*}
The Cayley transformation gives the equivalent Stein equation
$U_{j+1}=E_0^{\top}U_{j+1}E_0+P_0$, which is solved by
\begin{align}\label{eq:squared-smith}
 E_{l+1}=E_l^2,\;
 P_{l+1}=P_l+E_l^{\top}P_l E_l
 \to U_{j+1}.
\end{align}
This is the squared Smith iteration~\cite{smit:1968,hull:2018},
the specialization of SDA to a Lyapunov equation.
Stability of $L_j$ ensures $\rho(E_0)<1$ and convergence of the inner
iteration. The warm start and stabilizing safeguard are described in
Remark~\ref{rem3.1}.

Algorithm~\ref{alg:comp_EK} summarizes the computation of a controller satisfying (C1)--(C3) and the recovery of the controller matrices in \eqref{eq2.12}.

\begin{algorithm}[tbp]
\begin{algorithmic}[1]
\REQUIRE The current frozen matrices $A,B_1^{\rm o},B_2,C_1^{\rm o},C_2$ at synthesis instant $t_k$; (C3) margin $\tau\in(0,1)$, strict factor $\kappa>1$, update factor $\eta>0$, maximum number of attenuation updates $l_{\max}\ge0$; optional accepted values $(\gamma_{k-1},X_{k-1},Y_{k-1})$.
\ENSURE Success/failure status; on success, controller matrices $(\widehat{A}_0,\widehat{B}_0,\widehat{C}_0)$ for the projected realization of \eqref{eq2.3} and accepted values $(\gamma_k,X_k,Y_k)$ for reuse at the next sampling instant.
    \STATE Form $\widetilde B_1=B_2B_2^\dagger B_1^{\rm o}$ and
    $\widetilde C_1=C_1^{\rm o}C_2^\dagger C_2$;
    \STATE Compute $(\gamma_b,\gamma_c)$ by \eqref{eq3.19}. At the first state set $\gamma\leftarrow\kappa\max\{\gamma_b,\gamma_c\}$; thereafter set $\gamma\leftarrow\max\{\gamma_{k-1},\kappa\gamma_b,\kappa\gamma_c\}$;
    \STATE Form $G_\gamma,H_\gamma$ by \eqref{eq3.21} and compute $X(\gamma),Y(\gamma)$ for the CAREs
    \begin{align*}
        &A^{\top}X+XA-XG_\gamma X+\widetilde C_1^{\top}\widetilde C_1=0,\\
        &AY+YA^{\top}-YH_\gamma Y+\widetilde B_1\widetilde B_1^{\top}=0,
    \end{align*}
    respectively, using the selected solver settings and the Newton initialization in Remark~\ref{rem3.1}; return failure if either solve fails;
    \STATE Compute $\rho(\gamma) = \rho(X(\gamma) Y(\gamma))$ and set $l=0$;
    \WHILE{$\rho(\gamma)/\gamma^2 > 1-\tau$ and $l<l_{\max}$}
    \STATE Set $\gamma\leftarrow(1+\eta)\max\{\gamma,\sqrt{\rho(\gamma)/(1-\tau)}\}$
    \STATE Recompute $G_\gamma,H_\gamma$ by \eqref{eq3.21} and solve the two CAREs above with the selected solver settings and the Newton safeguard in Remark~\ref{rem3.1}; return failure if either solve fails;
    \STATE Compute $\rho(\gamma) = \rho(X(\gamma) Y(\gamma))$ and
    set $l=l+1$;
    \ENDWHILE
    \STATE Return failure if the (C3) margin is not met;
    \STATE Compute $\widehat{B}_0=(I-\gamma^{-2}YX)^{-1}YC_2^{\top}$;
    \STATE Set $\widehat{A}_0=W_X-\widehat{B}_0C_2$ and
    $\widehat{C}_0=-B_2^{\top}X$, where $W_X=A-G_\gamma X$.
    \STATE Set $(\gamma_k,X_k,Y_k)\leftarrow(\gamma,X,Y)$ and return success with $(\widehat{A}_0,\widehat{B}_0,\widehat{C}_0,\gamma_k,X_k,Y_k)$.
\end{algorithmic}
\caption{Computation of $H_\infty$-controller $\mathsf{K}=(\widehat{A}_0,\widehat{B}_0,\widehat{C}_0)$ for the projected realization of \eqref{eq2.3}}
\label{alg:comp_EK}
\end{algorithm}

\begin{Remark}\label{rem3.1}

At sampling instant $k$, the implementation initializes $\gamma_k$ from the accepted value $\gamma_{k-1}$ and reuses $X_{k-1}$ and $Y_{k-1}$ as Newton starting values. At the first sampling instant, or whenever a previous solution is unavailable or is not stabilizing for the current CARE at the selected $\gamma$, Newton is initialized with a stabilizing \textbf{icare} solution.

\end{Remark}
\section{Kalman state estimation}\label{sec:kalman-estimation}

The preceding state-dependent controller synthesis requires a state
estimate at which its coefficient matrices are evaluated. We obtain this
estimate from sampled measurements using a discrete Kalman filter.
At time $t_k$, let $\hat\bx_k=\hat\bx_{k|k}\in\mathbb R^n$ and
$P_k=P_{k|k}\in\mathbb R^{n\times n}$ denote the posterior state estimate
and its covariance. Initialize them with given $\hat\bx_0$ and
$P_0=P_0^\top\ge0$.

For an SDC model frozen at $\hat\bx_k$, write
$A_k=A(\hat\bx_k)$ and $B_{2,k}=B_2(\hat\bx_k)$, and let
$\bb_k$ be a known affine term that may include a prescribed nominal
input. Over a sampling interval of length $h>0$, set
$M_k=I_n-\tfrac h2A_k$. With $M_k$ nonsingular, trapezoidal
discretization gives
\begin{align}\label{eq4.3}
 F_k=M_k^{-1}(I_n+\tfrac h2A_k),
\end{align}
and
\begin{align*}
 B_{2,d,k}=hM_k^{-1}B_{2,k},\qquad
 \bd_k=hM_k^{-1}\bb_k.
\end{align*}
Let $\bu_k$ denote the known effective input over this interval. With
measurement matrix $C_2\in\mathbb R^{p_2\times n}$, the discrete model is
\begin{align*}
 \bx_{k+1}&=F_k\bx_k+\bd_k+B_{2,d,k}\bu_k+\boldsymbol\varepsilon_k,\\*
 \by_k&=C_2\bx_k+\boldsymbol\nu_k,
\end{align*}
where $\boldsymbol\varepsilon_k$ and $\boldsymbol\nu_k$ are independent
zero-mean white noise sequences with covariances $W_k\ge0$ and $V>0$.
For continuous process-noise intensity $W_c\ge0$, the corresponding
trapezoidal approximation is
\begin{align*}
 W_k=hM_k^{-1}W_cM_k^{-\top}.
\end{align*}

The prediction and correction recursion~\cite{simo:2006} is
\begin{subequations}\label{eq4.4}
\begin{align}
 \hat\bx_{k+1|k}&=F_k\hat\bx_k+\bd_k+B_{2,d,k}\bu_k,\\
 P_{k+1|k}&=F_kP_kF_k^\top+W_k,\\
 \mathcal K_{k+1}&=P_{k+1|k}C_2^\top
 (C_2P_{k+1|k}C_2^\top+V)^{-1},\\
 \hat\bx_{k+1}&=\hat\bx_{k+1|k}
 +\mathcal K_{k+1}(\by_{k+1}-C_2\hat\bx_{k+1|k}),\\
 P_{k+1}&=(I_n-\mathcal K_{k+1}C_2)P_{k+1|k}
 (I_n-\mathcal K_{k+1}C_2)^\top\notag\\*
 &\quad+\mathcal K_{k+1}V\mathcal K_{k+1}^\top.
\end{align}
\end{subequations}
Here $P_{k+1|k}$ is the prior covariance and $\mathcal K_{k+1}$ is
the Kalman gain. The covariance and gain are updated recursively,
avoiding an additional estimator CARE.

\section{Applications in aerial vehicle design}

We apply the control framework developed in Sections~2--4 to an F-16
aircraft and a quadrotor. The aircraft examples address recovery to steady
flight and speed tracking; the quadrotor examples address position tracking.

\subsection{F-16 model for regulation and speed tracking}\label{sec:f16-model}
The F-16 is controlled through engine thrust and the deflections of the
elevator, aileron, and rudder. To describe its motion, define
\begin{align}
 \mathbf r&=[x,y,z]^\top,&
 \mathbf v_b&=[u,v,w]^\top,\notag\\
 \boldsymbol\omega&=[p,q,r]^\top,&
 \boldsymbol\alpha&=[\phi,\theta,\psi]^\top.\label{eq6.1}
\end{align}
Here $\mathbf r$ is the position of the aircraft center of mass in the
earth frame; its vertical coordinate $z$ is altitude. The vector $\mathbf v_b$ is its velocity relative to the
earth frame, expressed in the body frame, and $\boldsymbol\omega$ is
the angular velocity of the body relative to the earth frame, expressed
in the body frame. The attitude vector $\boldsymbol\alpha$ consists of
the roll angle $\phi$, pitch angle $\theta$, and yaw angle $\psi$,
which describe the body orientation relative to the earth frame.
The twelve-state vector and the actuator inputs actually applied to the
aircraft are
\begin{equation}\label{eq:f16-state-input}
 \bx=[\mathbf r^\top,\boldsymbol\alpha^\top,\mathbf v_b^\top,\boldsymbol\omega^\top]^\top,\;
 \bu_{\mathrm p}=[T,\delta_e,\delta_a,\delta_r]^\top.
\end{equation}
Here $\bu_{\mathrm p}$ collects the physical actuator settings: $T$ is
the thrust in newtons, and $\delta_e$, $\delta_a$, and
$\delta_r$ are the elevator, aileron, and rudder deflection angles,
respectively, in radians. Positions and body velocities use metres and
metres per second, and attitude angles and angular rates use radians
and radians per second.

\begin{figure}[tbp]
 \centering
 \includegraphics[width=0.94\textwidth]{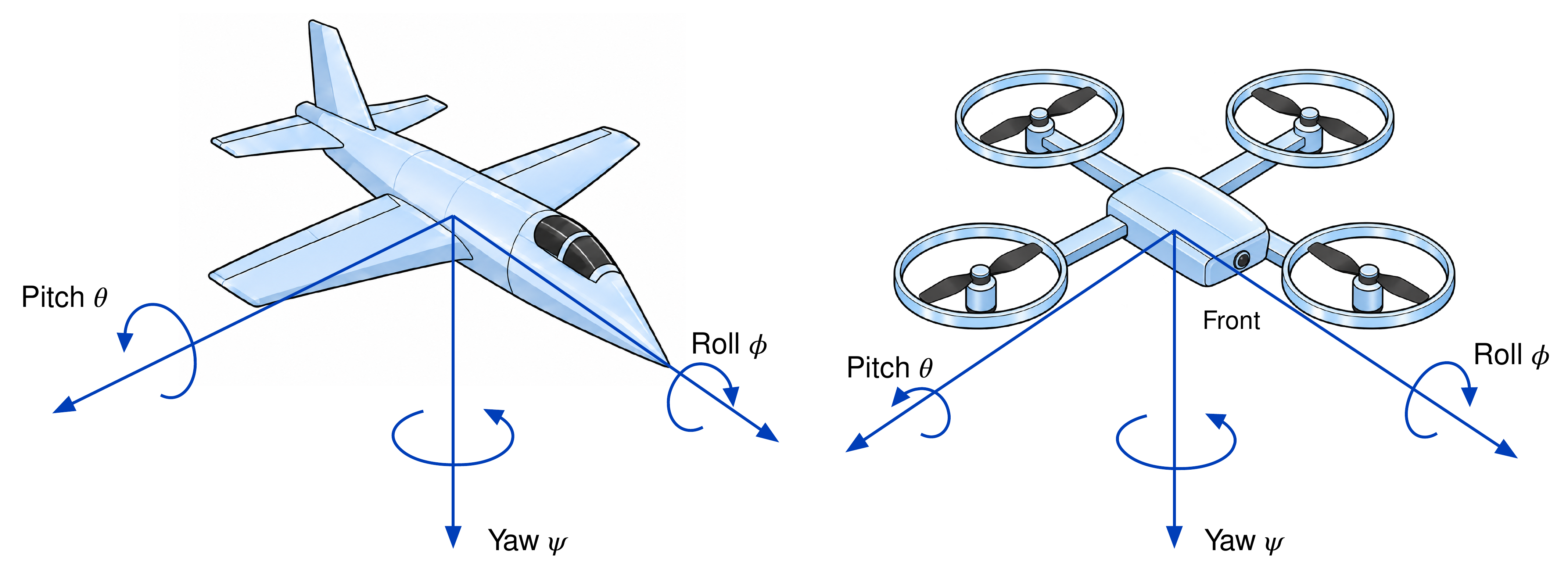}
 \caption{Schematic roll, pitch, and yaw conventions for the F-16 aircraft
 (left) and quadrotor (right) examples. AI-generated illustration.}
 \label{fig6.1}
\end{figure}

The elementary rotations are
\begin{subequations}\label{eq6.2}
\begin{align}
 \mathcal R_x(\phi)&=\begin{bmatrix}1&0&0\\0&\cos\phi&-\sin\phi\\0&\sin\phi&\cos\phi\end{bmatrix},\\
 \mathcal R_y(\theta)&=\begin{bmatrix}\cos\theta&0&\sin\theta\\0&1&0\\-\sin\theta&0&\cos\theta\end{bmatrix},\\
 \mathcal R_z(\psi)&=\begin{bmatrix}\cos\psi&-\sin\psi&0\\\sin\psi&\cos\psi&0\\0&0&1\end{bmatrix}.
\end{align}
\end{subequations}
Write $\mathcal R=\mathcal R_z\mathcal R_y\mathcal R_x$. The navigation
equation is $\dot{\mathbf r}=\mathcal R\mathbf v_b$, and the Euler-angle
kinematics are
\begin{subequations}\label{eq6.5}
\begin{align}
 \dot\phi&=p+q\sin\phi\tan\theta+r\cos\phi\tan\theta,\\
 \dot\theta&=q\cos\phi-r\sin\phi,\\
 \dot\psi&=(q\sin\phi+r\cos\phi)\sec\theta.\label{eq:dynamic_psi}
\end{align}
\end{subequations}
Let $m$ be the aircraft mass, $\mathcal I$ its body-frame inertia tensor,
and $\be_1,\be_3$ the first and third coordinate unit vectors. With
aerodynamic force $\mathbf F_a$ and moment $\mathbf M_a$, and external
force and moment disturbances $\mathbf f_w$ and $\boldsymbol\tau_w$,
the Newton--Euler equations are
\begin{align}\label{eq6.6}
 m(\dot{\mathbf v}_b+\boldsymbol\omega\times\mathbf v_b)
 &=-mg\mathcal R^\top\be_3+\mathbf F_a+T\be_1+\mathbf f_w,\notag\\*
 \mathcal I\dot{\boldsymbol\omega}
 +\boldsymbol\omega\times\mathcal I\boldsymbol\omega
 &=\mathbf M_a+\boldsymbol\tau_w.
\end{align}
The aerodynamic force $\mathbf F_a$ and moment $\mathbf M_a$ are
computed using \texttt{F16Model.jl}~\cite{isrlab:F16Model}.

Together, these equations define the nonlinear plant
\begin{align}\label{eq6.12}
 \dot\bx=f(\bx,\bu_{\mathrm p})+E\bw_{\mathrm p},\qquad
 \bw_{\mathrm p}=[\mathbf f_w^\top,\boldsymbol\tau_w^\top]^\top.
\end{align}
Here $f$ collects the navigation, attitude, and force and moment equations
above. The matrix $E$ maps the external forces to translational
accelerations and the external moments to angular accelerations; its
blocks are given in Appendix~\ref{app:model-matrices}.

The control objective is to follow a desired motion $\bx_d(t)$, which
contains the desired position, attitude, body velocity, and angular
velocity in the same order as $\bx$. The feedforward input
$\bu_{\mathrm f}(t)$ is the thrust and set of surface deflections that
produce this motion in the absence of disturbances. Thus,
\begin{align}\label{eq:f16-moving-reference}
 \dot\bx_d(t)=f(\bx_d(t),\bu_{\mathrm f}(t)).
\end{align}
For regulation, the reference is steady straight-level flight, obtained
by balancing the forces and moments. For speed tracking, the thrust
varies about this trim while the surface deflections remain fixed:
\begin{align}\label{eq:f16-moving-feedforward}
 \bu_{\mathrm f}(t)
 =[T_\star+A_T\sin(\Omega_T t_k),\delta_{e,\star},
                  \delta_{a,\star},\delta_{r,\star}]^\top.
\end{align}
The starred quantities are the trim inputs, $A_T$ is the thrust-modulation
amplitude, and $\Omega_T$ is its frequency. With sampling period $h$, this
input is evaluated at $t_k=kh$ and held over $[t_k,t_{k+1})$ within the
actuator bounds. Integrating \eqref{eq:f16-moving-reference} with this
input gives the desired state trajectory.

Feedback corrects deviations of the aircraft from $\bx_d$ by adjusting
the thrust and surface deflections relative to $\bu_{\mathrm f}$.
Because these state and input components have different units and
magnitudes, positive diagonal matrices $D_x$ and $D_u$ set their
componentwise scales. The scaled tracking error is
\begin{align*}
 \be=D_x^{-1}(\bx-\bx_d).
\end{align*}
The controller output $\bu$ is a dimensionless feedback increment, so $D_u\bu$
is the corresponding adjustment to the four physical actuator inputs.
The resulting command is
\begin{align*}
 \bu_{\mathrm{cmd}}=\operatorname{sat}_{[\bu_{\min},\bu_{\max}]}
                 (\bu_{\mathrm f}+D_u\bu).
\end{align*}
Here $\operatorname{sat}$ clips each component between the lower and
upper physical actuator bounds $\bu_{\min}$ and $\bu_{\max}$,
listed in Appendix~\ref{app:model-matrices}.
With instantaneous updates, the aircraft receives
$\bu_{\mathrm p}=\bu_{\mathrm{cmd}}$; while a new command is being
computed, it retains the previous complete command. The actual applied
correction, expressed in the same scales, is
\begin{align*}
 \bu_a=D_u^{-1}(\bu_{\mathrm p}-\bu_{\mathrm f}).
\end{align*}
Similarly, a positive diagonal matrix $D_w$ sets the force and moment
disturbance scales, giving
\begin{align*}
 \bw_d=D_w^{-1}\bw_{\mathrm p},\qquad
 B_{1d}^{\rm o}=D_x^{-1}ED_w.
\end{align*}
The numerical scales are specified in Appendix~\ref{app:model-matrices}.

To express the tracking problem in the SDC form of Sections~2--3,
subtract the reference dynamics from the aircraft dynamics. The
$12\times12$ matrix $A_e(\bx,t)$ represents the change in the drift
caused by the state error:
\begin{align}\label{eq:f16-error-dynamics}
 A_e(\bx,t)\be
 =D_x^{-1}\!\left[f(\bx,\bu_{\mathrm f})-f(\bx_d,\bu_{\mathrm f})\right].
\end{align}
Its construction preserves the navigation and angular-rate blocks and
the invariance under horizontal translations. The input matrix describes
the local response to changes in thrust and surface deflections:
\begin{align}\label{eq:f16-input-tangent}
 B_2(\bx,t)=D_x^{-1}f_{\bu_{\mathrm p}}(\bx,\bu_{\mathrm f})D_u.
\end{align}
Here $f_{\bu_{\mathrm p}}$ is the derivative of $f$ with respect to the
four actuator inputs, evaluated at the feedforward input. The resulting
error dynamics are
\begin{align}\label{eq:f16-moving-error}
 \dot\be&=A_e\be+B_2\bu_a+B_{1d}^{\rm o}\bw_d
              +\boldsymbol\rho_u,\notag\\*
 \boldsymbol\rho_u
 &=D_x^{-1}\!\left[f(\bx,\bu_{\mathrm p})-f(\bx,\bu_{\mathrm f})\right]
      -B_2\bu_a.
\end{align}
The term $\boldsymbol\rho_u$ accounts for the nonlinear dependence on
the actuator inputs. Aircraft motion is propagated using the nonlinear
model \eqref{eq6.12}.

The measurements $\by$ contain the nine position, Euler-angle, and
body-velocity components, selected by $C=[I_9\ \ 0]$. Let
$\boldsymbol\sigma_y$ contain their measurement-noise standard deviations.
Scaling each measurement error by its standard deviation gives
\begin{align*}
 D_y&=\operatorname{diag}(\boldsymbol\sigma_y)^{-1},\qquad C_2=D_yCD_x,\\*
 \by_e&=D_y(\by-C\bx_d)=C_2\be+\bw_n,
\end{align*}
where $\bw_n$ is the scaled measurement noise. The performance output
before projection is $C_1^{\rm o}\be+D_{12}\bu$: its state weights act
on position, attitude, and velocity errors, and its control cost acts on
the feedback increment. These matrices are given in
Appendix~\ref{app:model-matrices}.

The six process disturbances and nine measurement-noise components form
the exogenous input $\bw=[\bw_d^\top,\bw_n^\top]^\top$ of Section~2.
Following the projection in Section~3, set
$\widetilde B_{1d}=P_BB_{1d}^{\rm o}$ and
$\widetilde C_1=C_1^{\rm o}P_C$. The complete channels are
\begin{align*}
 B_1^{\rm o}&=[B_{1d}^{\rm o},0_{12\times9}],\\*
 \widetilde B_1&=[\widetilde B_{1d},0_{12\times9}],\qquad
 D_{21}=[0_{9\times6},I_9].
\end{align*}
At synthesis instant $t_k$, let $\hat\bx_k$ be the Kalman posterior
converted back to physical state coordinates. This estimate and the
reference data determine $A=A_e(\hat\bx_k,t_k)$ and the other
coefficient matrices.
The state and measurement $(\bx,\by)$ in Section~2 correspond to
$(\be,\by_e)$ here. The two CAREs and controller recovery use
$(B_1,C_1)=(\widetilde B_1,\widetilde C_1)$ and retain all twelve states.
The central controller has internal state $\xi$, driven by the measured
tracking error, and produces the feedback increment introduced above:
\begin{align}\label{eq:f16-physical-controller}
 \dot\xi&=\widehat A_0\xi+\widehat B_0\by_e,\qquad
 \bu=\widehat C_0\xi.
\end{align}

\begin{samepage}
\subsection{Quadrotor model for position tracking}
We next consider tracking a three-dimensional spiral with the quadrotor
rigid-body model of~\cite{ChHu:2022}. The vectors $\boldsymbol\alpha$,
$\boldsymbol\omega$, $\mathbf v_b$, and $\mathbf r$ have the physical
meanings defined in Section~\ref{sec:f16-model}; here the vertical
coordinate $z$ increases in the direction of gravity. The state order
and applied actuator inputs are
\begin{equation}\label{eq6.15}
 \bx=[\boldsymbol\alpha^\top,\boldsymbol\omega^\top,\mathbf v_b^\top,\mathbf r^\top]^\top,\;
 \bu_{\mathrm p}=[T,\tau_x,\tau_y,\tau_z]^\top.
\end{equation}
\end{samepage}
Here $T$ is the total thrust in newtons, acting along the negative body
$z$-axis, and $\boldsymbol\tau=[\tau_x,\tau_y,\tau_z]^\top$ is the
applied body torque in newton metres.
With $\be_3=[0,0,1]^{\top}$, the body-frame Newton--Euler equations are
\begin{subequations}\label{eq6.16}
\begin{align}
 \dot{\mathbf r}&=\mathcal R\mathbf v_b,\label{eq6.16a}\\*
 m(\dot{\mathbf v}_b+\boldsymbol{\omega}\times\mathbf v_b)
  &=mg\mathcal R^{\top}\be_3-T\be_3+\mathbf f_w,\label{eq6.16b}\\*
 \mathcal I\dot{\boldsymbol{\omega}}
  +\boldsymbol{\omega}\times\mathcal I\boldsymbol{\omega}
  &=\boldsymbol{\tau}+\boldsymbol{\tau}_w,\label{eq6.16c}
\end{align}
\end{subequations}
where $\mathcal R=\mathcal R_z(\psi)\mathcal R_y(\theta)
\mathcal R_x(\phi)$ is given by \eqref{eq6.2} and
$\mathcal I=\operatorname{diag}(I_x,I_y,I_z)$.  The same 3--2--1
kinematic convention as in \eqref{eq6.5} gives
\begin{align}\label{eq6.17}
 \dot{\boldsymbol{\alpha}}
 =\begin{bmatrix}
 1&\sin\phi\tan\theta&\cos\phi\tan\theta\\
 0&\cos\phi&-\sin\phi\\
 0&\sin\phi\sec\theta&\cos\phi\sec\theta
 \end{bmatrix}\boldsymbol{\omega}.
\end{align}
To simplify the rotational equations, define the inertia coefficients
\begin{align*}
&c_1=(I_y-I_z)/I_x,&&c_2=(I_z-I_x)/I_y,\\*
&c_3=(I_x-I_y)/I_z,&&c_4=1/I_x,\\*
&c_5=1/I_y,&&c_6=1/I_z.
\end{align*}
Then expanding \eqref{eq6.16b}--\eqref{eq6.16c} gives
\begin{subequations}\label{eq6.18}
\begin{align}
&\dot p=c_1qr+c_4(\tau_x+\tau_{w,x}),\\
&\dot q=c_2pr+c_5(\tau_y+\tau_{w,y}),\\
&\dot r=c_3pq+c_6(\tau_z+\tau_{w,z}),\label{eq6.18b}\\
&-mg\sin\theta+f_{w,x}=m(\dot u+qw-rv),\\
&mg\sin\phi\cos\theta+f_{w,y}=m(\dot v-pw+ru),\\
&mg\cos\phi\cos\theta-T+f_{w,z}=m(\dot w+pv-qu).
\end{align}
\end{subequations}
The force and moment disturbances form
$\bw_{\mathrm p}=[\mathbf f_w^\top,\boldsymbol\tau_w^\top]^\top$.

For a prescribed reference position $\mathbf r_d$, attitude
$\mathcal R_d$, and body angular velocity $\boldsymbol\omega_d$,
the feedforward thrust and torque satisfy
\begin{align*}
 T_d\mathcal R_d\be_3&=m(g\be_3-\ddot{\mathbf r}_d),\\*
 \boldsymbol\tau_d&=\mathcal I\dot{\boldsymbol\omega}_d
 +\boldsymbol\omega_d\times\mathcal I\boldsymbol\omega_d.
\end{align*}
The reference body velocity is
$\mathbf v_{b,d}=\mathcal R_d^\top\dot{\mathbf r}_d$, and its angular
kinematics follow \eqref{eq6.17}. We use the error and input definitions
of Section~\ref{sec:f16-model}, with
$D_x=I_{12}$, $D_u=I_4$, $D_w=I_6$, and $D_y=I_9$. Thus
$\be=\bx-\bx_d$, $\bw_d=\bw_{\mathrm p}$, and the command is
\begin{align}\label{eq:quad-input-split}
 \bu_{\mathrm{cmd}}(t)&=\bu_{\mathrm f}(t)+\bu(t),\\*
 \bu_{\mathrm f}(t)&=[T_d(t),\boldsymbol\tau_d(t)^\top]^\top.\notag
\end{align}
The feedback increment $\bu$ adjusts the thrust and torques needed for
the reference motion. The applied correction is
$\bu_a=\bu_{\mathrm p}-\bu_{\mathrm f}$, with the same command update
and hold rule as for the F-16. Instantaneous updates give $\bu_a=\bu$.

\begin{samepage}
For the state order in \eqref{eq6.15}, separating the constant gravity
term and the feedforward input yields the affine SDC model
\begin{align}\label{eq6.19}
 \dot\bx=A(\bx)\bx+\bb_{\mathrm f}(t)+B_{1d}^{\rm o}\bw_d+B_2\bu_a,
\end{align}
where
\begin{align*}
 \bb_{\mathrm f}(t)&=g\be_9+B_2\bu_{\mathrm f}(t),\\*
 \be_9&=[0,0,0,0,0,0,0,0,1,0,0,0]^\top.
\end{align*}
\end{samepage}
In the filter of Section~\ref{sec:kalman-estimation}, $\bb_k=g\be_9$
and $\bu_k$ is the effective value of the actual physical input
$\bu_{\mathrm p}$ over the sampling interval. The known forcing can
equivalently be written as $\bb_{\mathrm f}+B_2\bu_a$.
Here $A(\bx)\bx+g\be_9$ is the uncontrolled drift in
\eqref{eq6.16}--\eqref{eq6.18}. At hover,
$\bu_{\mathrm f}=[mg,0,0,0]^\top$ and $\bb_{\mathrm f}=0$.
Subtracting the reference rate gives
\begin{align*}
 \dot\be&=A(\bx)\be+B_2\bu_a+B_{1d}^{\rm o}\bw_d
                  +\boldsymbol\delta_{\mathrm{tr}}(\bx,t),\\*
 \boldsymbol\delta_{\mathrm{tr}}(\bx,t)
          &=A(\bx)\bx_d+\bb_{\mathrm f}(t)-\dot\bx_d.
\end{align*}
For an exact nominal solution, this remainder equals
$[A(\bx)-A(\bx_d)]\bx_d$.

The measured components are attitude, body angular velocity, and
position. Since $D_x$ and $D_y$ are identities, $C_2=C$ and
$\by_e=\by-C\bx_d$. Synthesis follows the F-16 construction with
$A_e(\bx,t)=A(\bx)$ and the quadrotor disturbance, performance, and
measurement matrices, including the same completion and projection of
the disturbance channel. The controller is
\begin{align}\label{eq:quad-feedback-increment}
 \dot\xi&=\widehat A_0\xi+\widehat B_0\by_e,\qquad
 \bu=\widehat C_0\xi.
\end{align}
The coefficient matrices are evaluated at the Kalman posterior state.
The performance output is $\bz=\widetilde C_1\be+D_{12}\bu$, with
nonzero position-error weights.

The physical parameters, matrices $A(\bx)$, $B_{1d}^{\rm o}$, $B_2$,
performance and measurement channels, and projected channels used in the two CAREs are
given in Appendix~\ref{app:model-matrices}.
\section{Numerical experiments}\label{sec:numerical-experiments}

The experiments compare \texttt{icare}, SDA, and Newton on F-16
regulation and speed tracking, and nominal and faster quadrotor spiral
tracking. Within each test, the solvers share the model, reference,
initial conditions, and noise samples. Computations use MATLAB R2023b
Update~1 on an Apple M2 Pro with 16~GiB memory and a single computational
thread.

Regulation and nominal tracking use instantaneous control updates.
The additional tracking tests simulate limited onboard computation with
a common resource model based on measured complete-task times. Tasks run sequentially: during computation,
the previous complete physical input, including feedforward, is held;
the new command is applied upon completion.
In Figures~\ref{fig:f16-moving-tracking} and~\ref{fig:quad-resource},
``Zero-delay'' denotes SDA with instantaneous control updates.

\subsection{F-16 regulation and speed tracking}\label{sec:f16-experiments}

\subsubsection{Recovery to straight-level flight}\label{sec:aircraft-regulation}
The reference describes flight at a constant speed of $183.3$ m/s
and altitude $3048$ m.
Solving all six force and moment balances gives the trim state and input.
Its nonzero body vertical velocity is consistent with zero climb rate
through $\dot{\mathbf r}_d=\mathcal R_d\mathbf v_{b,d}$.
The plant starts from this trim plus
\begin{align*}
 \Delta\bx_0=(&0,0,0,2^\circ,1^\circ,0,\\
             &2,0.4,-0.4,0.01,-0.004,0.002)^\top .
\end{align*}
The estimator starts at trim and $\xi(0)=0$. The scales, noise
statistics, and covariance initialization are specified in
Appendix~\ref{app:model-matrices}.

The experiment covers $30$ s with sensor period $h=0.005$ s.
The nonlinear plant and filter mean use RK4 substeps of at most
$0.0025$ s, and the dynamic controller uses the trapezoidal update.
The complete feedforward-plus-feedback input is constrained by the
source actuator bounds. Full-state NRMSE is the root mean square of the
twelve scaled error components $D_x^{-1}(\bx-\bx_d)$ over all $6001$
state samples, including the initial sample. The airspeed error is
$\|\mathbf v_b\|_2-\|\mathbf v_{b,d}\|_2$.

Figure~\ref{fig:aircraft-regulation} shows the airspeed, altitude, and
pitch errors during recovery. The three solver trajectories overlap
at plotting resolution. Table~\ref{tab:f16-regulation} summarizes the
physical response. The initial transient includes
thrust saturation, after which the aircraft approaches the moving
straight-level reference.

\begin{table}[tbp]
 \centering
 \caption{F-16 straight-level recovery with instantaneous control updates.}
 \label{tab:f16-regulation}
 \small
 \begin{tabular}{lr}
  \toprule
  Metric & Value\\
  \midrule
  Position RMSE (m) & 3.3335\\
  Airspeed RMSE (m/s) & 0.4862\\
  Full-state NRMSE & 0.08730\\
  Maximum altitude deviation (m) & 5.8163\\
  Final position error norm (m) & 0.2171\\
  Final altitude error (m) & 0.002541\\
  Final climb rate (m/s) & $-0.000565$\\
  Final airspeed error (m/s) & $-0.021475$\\
  \bottomrule
 \end{tabular}
\end{table}

\begin{figure}[tbp]
 \centering
 \includegraphics[width=\textwidth]{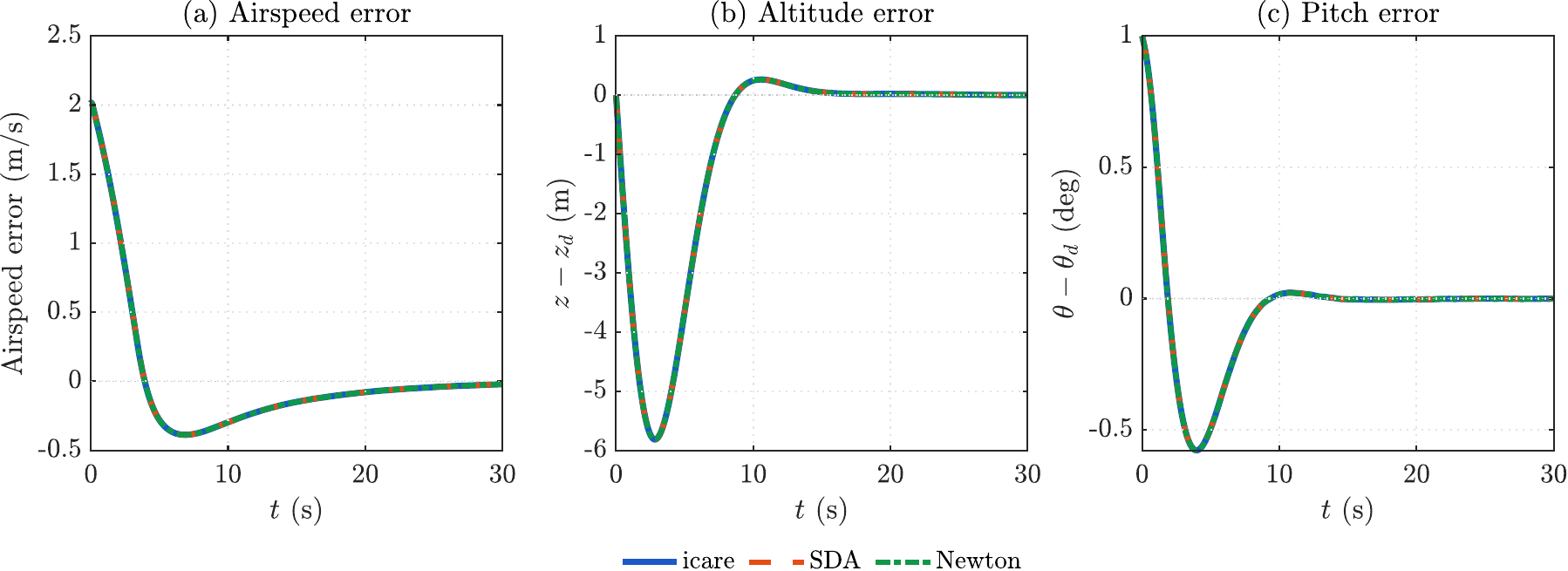}
 \caption{F-16 recovery to straight-level flight.}
 \label{fig:aircraft-regulation}
\end{figure}

\subsubsection{Speed tracking with measured computation delays}
\label{sec:f16-speed-tracking}
The reference is generated offline from
\eqref{eq:f16-moving-reference}--\eqref{eq:f16-moving-feedforward},
starting from straight-level trim, with $A_T=4500$ N and
$\Omega_T=\pi$ rad/s. The three surface inputs remain at trim.
The complete twelve-state reference follows the position, attitude,
velocity, and rate evolution induced by the sampled thrust modulation.

The experiment covers $20$ s with sensor period $h=0.005$ s.
The plant and estimator start on the reference and $\xi(0)=0$;
state scales and noise statistics are shared with the regulation test.
The nonlinear plant and filter mean use RK4 substeps of at most
$0.0025$ s, split at command-publication times.
For prediction to the next synthesis time, the covariance uses the
frozen plant Jacobian computed at the preceding synthesis instant,
and the controller state uses a trapezoidal update.

The airspeed RMSE compares $\|\mathbf v_b\|_2$ with
$\|\mathbf v_{b,d}\|_2$ over all $4001$ state samples, including the
initial sample. Pitch errors use the same moving reference.
Table~\ref{tab:f16-moving-tracking} and
Figure~\ref{fig:f16-moving-tracking} report one paired run.
All methods complete all six paired runs. Relative to \texttt{icare},
SDA and Newton reduce airspeed RMSE by $11.6$--$18.7\%$ and
$8.5$--$19.9\%$, respectively. Both methods also reduce full-horizon
pitch RMSE in every pair.

\begin{table}[tbp]
 \centering
 \caption{F-16 speed-tracking errors with computation delays.}
 \label{tab:f16-moving-tracking}
 \small
 \begin{tabular}{lrr}
  \toprule
  Solver & Airspeed RMSE & Pitch RMSE\\
  & (mm/s) & ($10^{-3}$ deg)\\
  \midrule
  \texttt{icare} & 76.1561 & 2.6582\\
  SDA & 63.3754 & 1.4236\\
  Newton & 64.9919 & 1.2831\\
  \bottomrule
 \end{tabular}
\end{table}

\begin{figure}[tbp]
 \centering
 \begin{subfigure}[t]{0.48\textwidth}
  \centering
  \includegraphics[width=\linewidth]{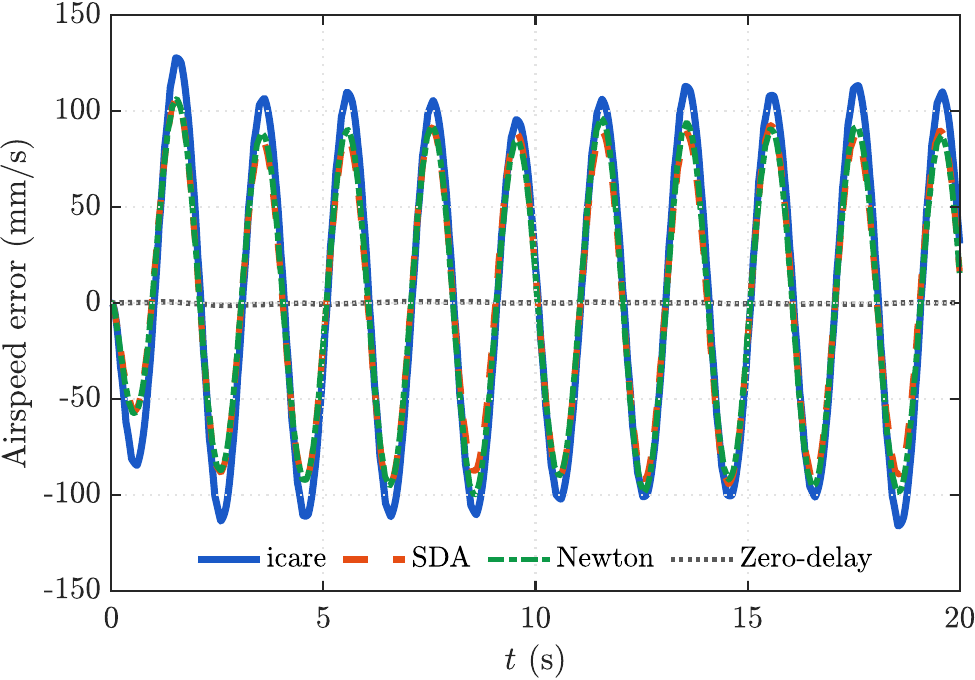}
  \caption{Airspeed error.}
 \end{subfigure}\hfill
 \begin{subfigure}[t]{0.48\textwidth}
  \centering
  \includegraphics[width=\linewidth]{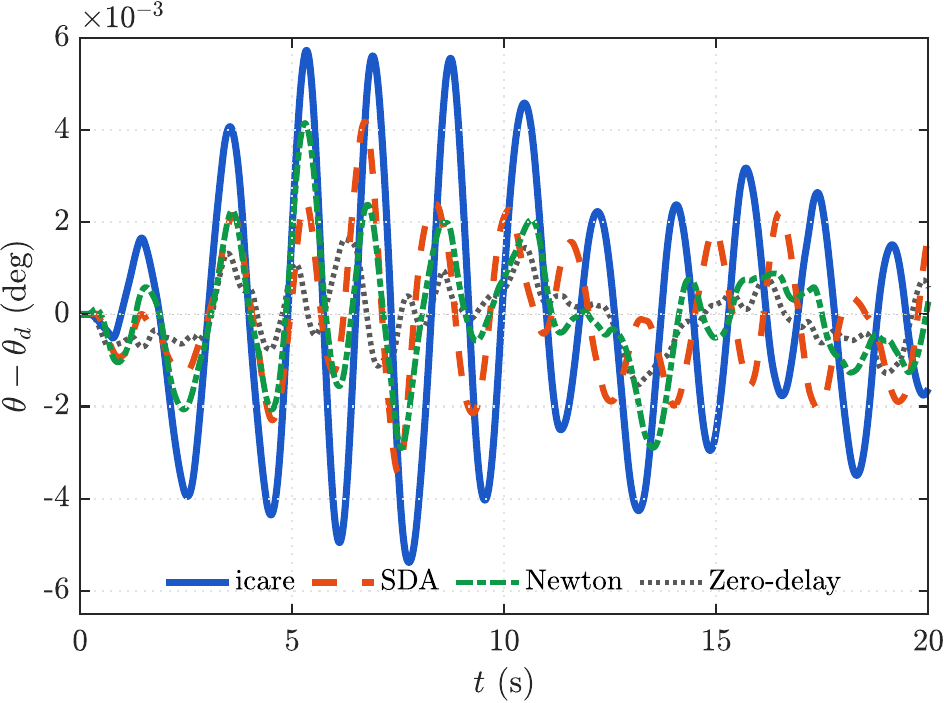}
  \caption{Pitch error.}
 \end{subfigure}
 \caption{F-16 speed tracking with computation delays.}
 \label{fig:f16-moving-tracking}
\end{figure}

\subsection{Quadrotor position tracking}\label{sec:quadrotor-experiments}

\subsubsection{Nominal spiral tracking}\label{sec:quadrotor-baseline}
\begin{figure}[tbp]
 \centering
 \includegraphics[width=0.76\textwidth]{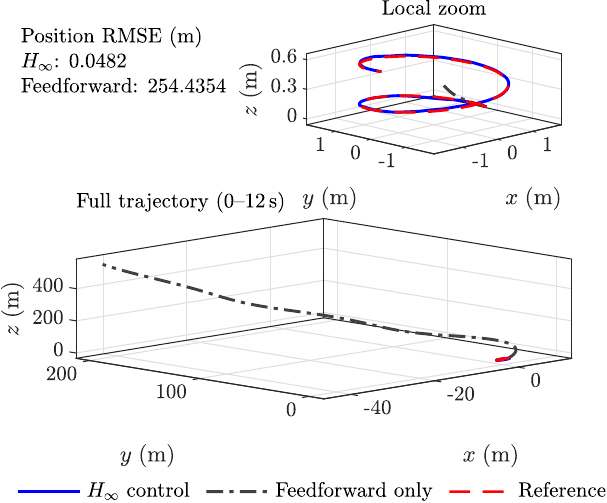}
 \caption{Quadrotor spiral tracking with instantaneous control updates.}
 \label{fig:quadrotor-baseline}
\end{figure}
The quadrotor tracks a spiral of radius 1.5~m, angular rate $0.25\pi$~rad/s, and vertical-coordinate rate 0.05~m/s for 12~s with $h=0.002$~s. The reference has zero yaw; its attitude, body velocity, and angular rates are constructed consistently with the prescribed position trajectory. The feedforward input in \eqref{eq:quad-input-split} is computed in advance, using sampled numerical derivatives for the reference angular motion, and enters the known model term $\bb_{\mathrm{f}}(t)$. The nine components $[\phi,\theta,\psi,p,q,r,x,y,z]$ are measured with noise, while the Kalman filter reconstructs the full twelve-state vector. All three solvers give the same position RMSE of \QuadrotorPositionRMSE~m, evaluated against the reference at the same post-update sample times. The maximum entrywise difference over all pairs of solver trajectories is $\QuadrotorMaxTrajectoryDifference$.

A paired feedforward-only run sets $\bu=0$ in \eqref{eq:quad-input-split}, giving the physical input $\bu_{\mathrm{p}}=\bu_{\mathrm{f}}$. The model, feedforward input, initial state, reference, Euler step, and process-disturbance samples are identical to those of the controlled run. Both runs start with zero attitude and angular rates, which differ from the moving reference state. Over the same 6000 post-update samples spanning 12~s, the three-dimensional position RMSE is \QuadFeedforwardRMSE~m for feedforward only, compared with \QuadrotorPositionRMSE~m with the dynamic $H_\infty$ feedback and Kalman estimator. Figure~\ref{fig:quadrotor-baseline} shows the full reference, controlled (Newton), and feedforward-only trajectories, with a local enlargement of the tracking region.

\subsubsection{Faster spiral tracking under limited computational resources}
\label{sec:quad-resource}
\begin{figure}[tbp]
 \centering
 \begin{subfigure}[t]{0.48\textwidth}
  \centering
  \includegraphics[width=\linewidth]{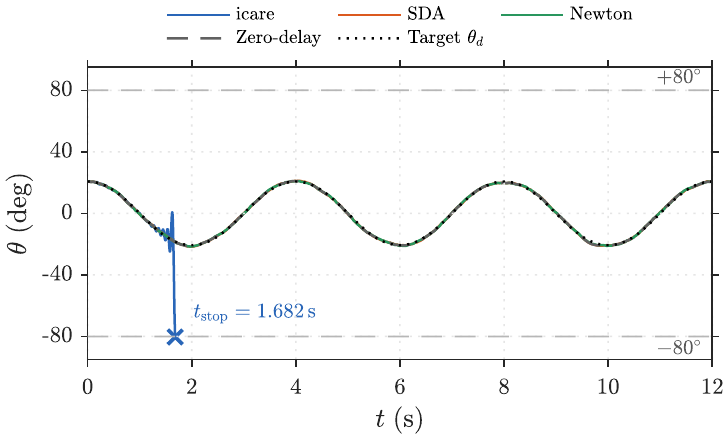}
  \caption{Pitch response.}
 \end{subfigure}\hfill
 \begin{subfigure}[t]{0.48\textwidth}
  \centering
  \includegraphics[width=\linewidth]{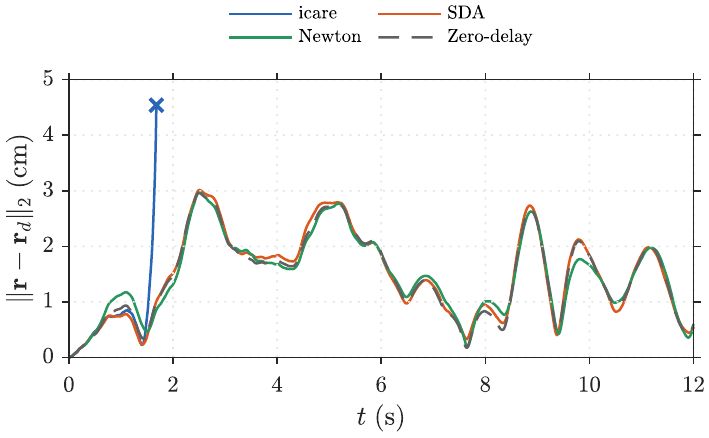}
  \caption{Position-error norm.}
 \end{subfigure}
 \caption{Quadrotor tracking under limited onboard computation. Crosses indicate \texttt{icare} termination.}
 \label{fig:quad-resource}
\end{figure}
The preceding example compares the three solvers with instantaneous control updates.
We now increase the angular rate of the spiral reference and simulate limited
onboard computation to examine how solver computation time affects tracking
performance. The reference is
\begin{equation}
 \mathbf r_d(t)=
 \begin{bmatrix}
  1.5\cos(\pi t/2)\\1.5\sin(\pi t/2)\\0.05t
 \end{bmatrix}\mathrm{m},\qquad 0\le t\le12~\mathrm{s}.
 \label{eq:quad-resource-reference}
\end{equation}
It completes three turns, twice the angular rate of the nominal example.
Positive $z$ follows gravity, and the yaw reference is zero. The remaining
reference states and feedforward input are constructed from this motion.

The sensor period is $h=1$~ms, and RK4 plant substeps are at most
$0.25$~ms, split at command publications. The controller retains its
Euler update and the estimator its discrete Kalman recursion. The plant
state and estimate start at the complete moving reference state,
$\bx(0)=\widehat\bx(0)=\bx_d(0)$, with $\xi(0)=0$, $P_0=I_{12}$,
and initial physical input $\bu_{\mathrm f}(0)$.
Synthesis weights, attenuation settings, and noise specifications are
unchanged; paired noise arrays are generated on the present sensor grid.
For completed runs, position RMSE is
\begin{equation}
 E_r=\left(\frac{1}{N}\sum_{k=1}^{N}
  \|\mathbf r(t_k)-\mathbf r_d(t_k)\|_2^2\right)^{1/2},
 \qquad N=12000.
 \label{eq:quad-resource-rmse}
\end{equation}

In the representative run, Figure~\ref{fig:quad-resource} shows
SDA and Newton following the reference while \texttt{icare} reaches
the predefined pitch stopping boundary $|\theta|>80^\circ$.
The attitude response reveals the loss of tracking before a large
position error develops.

This contrast persists across all six paired groups in
Table~\ref{tab:quad-resource-repeats}: SDA and Newton complete every
run, while \texttt{icare} stops early. The experimental configuration
is held fixed across these comparisons.

\begin{table}[tbp]
 \centering
 \caption{Quadrotor tracking under limited onboard computation.}
 \label{tab:quad-resource-repeats}
 \small
 \setlength{\tabcolsep}{4pt}
 \begin{tabular}{@{}lccc@{}}
  \toprule
  \multirow{2}{*}{Run} & \multirow{2}{*}{\shortstack{\texttt{icare}\\Stop time (s)}} & \multicolumn{2}{c}{Position RMSE (cm)}\\
  \cmidrule(l){3-4}
  & & SDA & Newton\\
  \midrule
  1 & $1.682^{\mathrm P}$ & 1.677 & 1.640\\
  2 & $3.485^{\mathrm P}$ & 1.607 & 1.683\\
  3 & $3.176^{\mathrm P}$ & 1.726 & 1.640\\
  4     & $0.541^{\mathrm P}$ & 1.503 & 1.651\\
  5     & $1.552^{\mathrm P}$ & 1.636 & 1.643\\
  6     & $5.939^{\mathrm R}$ & 1.396 & 1.335\\
  \bottomrule
 \end{tabular}
 \par\smallskip
 \begin{minipage}{\linewidth}
  \footnotesize
  P: pitch-envelope crossing; R: CARE residual acceptance failure.
 \end{minipage}
\end{table}

\subsection{Riccati computation time}\label{sec:care-timing}

For each vehicle model, the three solvers are evaluated on the same
sequence of states and corresponding controller CAREs. The timing covers
both CARE solves at each update. Table~\ref{tab:timing-results}
reports the CPU time per update and the speedup relative to
\texttt{icare}.

SDA and Newton require substantially less CPU time than \texttt{icare}
for both models. Consecutive states yield related CAREs, and Newton
reuses the preceding accepted solutions as initial guesses. SDA applies
structure-preserving doubling to each CARE, while Newton uses doubling
for its inner Lyapunov equations.

\begin{table}[htbp]
    \centering
    \caption{CPU time and speedup per update for the two controller CAREs.}
    \label{tab:timing-results}
    {
    \input{generated/paper_timing_table.tex}
    }
\end{table}

\section{Conclusions}

This paper presented an SDRE workflow based on projected disturbance and performance channels, explicit large-$\gamma$ feasibility, repeated SDA or safeguarded Newton--Kleinman CARE solves, and Kalman state estimation. Classical Riccati monotonicity supports the attenuation update enforcing (C3), and the stable invariant-subspace construction connects the projected CAREs to the central $H_\infty$ controller. The aircraft recovery and nominal quadrotor experiments give essentially solver-independent accuracy with instantaneous command updates. With measured computation delays, SDA and Newton reduce airspeed and pitch errors in all six paired aircraft tracking runs. In the faster quadrotor simulation with limited computational resources, SDA and Newton complete all six paired runs while runs using \texttt{icare} terminate early.

\appendix
\section{Parameters and coefficient matrices for aerial-vehicle models}
\label{app:model-matrices}

\paragraph{F-16 aircraft}
Use the state order in \eqref{eq:f16-state-input}, with four
three-dimensional blocks. The mass and gravitational acceleration are
\begin{align*}
 m&=9295.440537\ {\rm kg},\qquad g=9.805416\ {\rm m\,s^{-2}}.
\end{align*}
The inertia in $\mathrm{kg\,m^2}$ is
{\setlength{\arraycolsep}{2pt}
\begin{align*}
 \mathcal I&=
 \begin{bmatrix}
 12874.847237&0&1331.413225\\
 0&75673.622968&0\\
 1331.413225&0&85552.112540
 \end{bmatrix}.
\end{align*}
}
The lower and upper input bounds $\bu_{\min}$ and $\bu_{\max}$ follow
the source actuator limits: thrust $1000$--$19000$ lbf, elevator
$\pm25^\circ$, aileron $\pm21.5^\circ$, and rudder $\pm30^\circ$.
Thrust and surface deflections are converted to newtons and radians,
respectively.
The scales and unprojected synthesis channels are
\begin{equation*}
\begin{aligned}
 D_x=\operatorname{diag}\bigl(&20,20,10,0.1,0.1,0.1,\\
                            &5,5,5,0.1,0.1,0.1\bigr),\\
 D_u=\operatorname{diag}\bigl(&10^4,0.05,0.05,0.05\bigr).
\end{aligned}
\end{equation*}
\begin{align*}
 E&=\begin{bmatrix}
 0_{6\times3}&0_{6\times3}\\
 m^{-1}I_3&0\\
 0&\mathcal I^{-1}
 \end{bmatrix},\qquad B_{1d}^{\rm o}=D_x^{-1}E D_w,\\
 D_w&=100I_6,\qquad \bw_d=D_w^{-1}\bw_{\mathrm p},\\
 C_x&=\operatorname{diag}(I_9,0_{3\times3}),\\
 C_1^{\rm o}&=\begin{bmatrix}C_x\\0_{4\times12}\end{bmatrix},\qquad
 D_{12}=\begin{bmatrix}0_{12\times4}\\I_4\end{bmatrix},\\
 C_2&=D_y C D_x,\qquad C=[I_9\ \ 0_{9\times3}].
\end{align*}
Each entry of $D_x$ and $D_u$ carries the physical unit of its component.
The first and last three disturbance scales in $D_w$ represent
$100$ N and $100$ N\,m, respectively. The performance output before
projection is $C_1^{\rm o}\be+D_{12}\bu$.
The projectors in \eqref{eq3.19proj} are computed from the current
$B_2$ and $C_2$; in particular $P_C=\operatorname{diag}(I_9,0_3)$.

For clarity, partition $\be$ into position, attitude, velocity, and
rate blocks with corresponding diagonal scales $D_r,D_\alpha,D_v,D_\omega$.
Let $\mathcal T(\boldsymbol\alpha)$ be the Euler map in \eqref{eq6.5}
and $\mathcal T_d=\mathcal T(\boldsymbol\alpha_d)$ its reference value.
The upper six rows of the structured drift matrix are
\begin{align*}
 A_e(1\!:\!6,:)=
 \begin{bmatrix}
 0&A_{r\alpha}&D_r^{-1}\mathcal R D_v&0\\
 0&A_{\alpha\alpha}&0&D_\alpha^{-1}\mathcal T D_\omega
 \end{bmatrix},
\end{align*}
where
\begin{align*}
 A_{r\alpha}\be_\alpha
 &=D_r^{-1}(\mathcal R-\mathcal R_d)\mathbf v_{b,d},\\
 A_{\alpha\alpha}\be_\alpha
 &=D_\alpha^{-1}(\mathcal T-\mathcal T_d)\boldsymbol\omega_d.
\end{align*}
The attitude blocks are obtained by averaging forward and reverse
coordinate divided differences. For the lower six rows, set
\begin{equation*}
 D_{v\omega}=\operatorname{diag}(D_v,D_\omega),\;
 \mathcal F(\boldsymbol\zeta)=D_{v\omega}^{-1}f_{7:12}(\boldsymbol\zeta,\bu_{\mathrm f}).
\end{equation*}
Construct $A_e$ separately for the coordinate orders $(3,\ldots,12)$
and $(12,\ldots,3)$. In each order, start from $\bx_d$ and successively
replace coordinates by their values in $\bx$. Let $\boldsymbol\zeta^-$
and $\boldsymbol\zeta^+$ be the states before and after replacing
coordinate $j$, respectively. The corresponding column is
\begin{align*}
 (A_e)_{7:12,j}
 =\frac{\mathcal F(\boldsymbol\zeta^+)
              -\mathcal F(\boldsymbol\zeta^-)}{e_j}.
\end{align*}
Average the matrices from the two orders. Columns for horizontal position remain
zero. At scaled coordinate differences below $10^{-7}$, centered
tangents use a scaled perturbation $2\times10^{-5}$.
The remaining secant residual is corrected within the attitude columns
for the kinematic rows and columns $3{:}12$ for the dynamic rows.
Denoting the selected row and column sets by $\mathcal J_r,\mathcal J_c$, respectively,
\begin{align*}
 (A_e)_{\mathcal J_r,\mathcal J_c}
 &\ \leftarrow\ (A_e)_{\mathcal J_r,\mathcal J_c}
 +\frac{\boldsymbol\eta_{\mathcal J_r}\be_{\mathcal J_c}^{\top}}
        {\|\be_{\mathcal J_c}\|_2^2},\\
 \boldsymbol\eta&=D_x^{-1}[f(\bx,\bu_{\mathrm f})-f(\bx_d,\bu_{\mathrm f})]-A_e\be.
\end{align*}
The update is applied for nonzero selected error and residual above
the floating-point roundoff threshold. This yields
\eqref{eq:f16-error-dynamics} to evaluation precision while preserving
the physical kinematic blocks. The same scaled perturbation is used
for the four centered input differences in \eqref{eq:f16-input-tangent}.

For the nine sensors, the physical standard deviations are
\begin{equation*}
\begin{aligned}
 \boldsymbol\sigma_y
  =(&0.03,0.03,0.03,\\
    &0.01^\circ,0.01^\circ,0.01^\circ,0.01,0.01,0.01)^\top .
\end{aligned}
\end{equation*}
The angular entries are converted to radians.
The filter uses coordinates $\bar\bx=D_x^{-1}\bx$ and scales each
measurement by the corresponding state scale. Its observation matrix
is $C=[I_9\ \ 0]$, used in place of $C_2$ in the correction equations of
Section~\ref{sec:kalman-estimation}. With
$\mathbf s_x=\operatorname{diag}(D_x)$, its measurement covariance is
\begin{align*}
 V=\operatorname{diag}(\boldsymbol\sigma_y./\mathbf s_{x,1:9})^2.
\end{align*}
Let $\mathbf q$ contain the componentwise process-noise intensity scales
and $\mathbf p_0$ the initial estimation-error standard deviations used
to initialize the filter covariance, both in physical coordinates. The covariance intensity and initialization are
\begin{align*}
 W_c&=\operatorname{diag}(\mathbf q./\mathbf s_x)^2,\qquad
 P_0=\operatorname{diag}(\mathbf p_0./\mathbf s_x)^2,
\end{align*}
where
\begin{align*}
 \mathbf q=(&0.01,0.01,0.01,0.005,0.005,0.005,\\*
            &0.1,0.1,0.1,0.02,0.02,0.02)^\top,\\
 \mathbf p_0=(&0.03,0.03,0.03,0.1^\circ,0.1^\circ,0.1^\circ,\\*
              &0.05,0.05,0.05,0.01,0.01,0.01)^\top .
\end{align*}
Here $./$ denotes componentwise division and angular values are expressed
in radians. State prediction follows the nonlinear dynamics. Covariance
propagation uses a local linearization with the discretization in
\eqref{eq4.3}, followed by the correction in \eqref{eq4.4}.
Physical disturbance samples have standard deviations of $10$ N for
each body-force component and $1$ N\,m for each moment component.
They are held on the sensor grid and are independent of measurement noise.

\paragraph{Quadrotor}
The numerical model uses $m=1\,\mathrm{kg}$, $g=9.8\,\mathrm{m\cdot s^{-2}}$,
\begin{align*}
I_x=I_y=0.01466\,\mathrm{kg\cdot m^2},\quad
I_z=0.02848\,\mathrm{kg\cdot m^2}.
\end{align*}
Partition the state in \eqref{eq6.15} into four three-dimensional blocks and
define
\begin{align*}
[\boldsymbol{\omega}]_\times&=
\begin{bmatrix}0&-r&q\\r&0&-p\\-q&p&0\end{bmatrix},\\
\mathcal T(\phi,\theta)&=
\begin{bmatrix}
1&\sin\phi\tan\theta&\cos\phi\tan\theta\\
0&\cos\phi&-\sin\phi\\
0&\sin\phi\sec\theta&\cos\phi\sec\theta
\end{bmatrix},\\
A_\omega(\boldsymbol{\omega})&=
\begin{bmatrix}
0&0&c_1q\\
0&0&c_2p\\
0&c_3p&0
\end{bmatrix}.
\end{align*}
The gravity factor $G(\phi,\theta)=[g_{ij}]\in\mathbb R^{3\times3}$
has the nonzero entries
\begin{align*}
&g_{12}=-\frac{g\sin\theta}{\theta},\quad
g_{21}=\frac{g\cos\theta\sin\phi}{\phi},\\
&g_{31}=-2g\,\phi^{-1}\sin^2(\phi/2)\times\cos^2(\theta/2),\\
&g_{32}=-2g\,\theta^{-1}\sin^2(\theta/2)\times\cos^2(\phi/2).
\end{align*}
These trigonometric quotients are
evaluated by continuous extension at zero.
The complete SDC matrix in \eqref{eq6.19} is then
\begin{align*}
A(\bx)=
\begin{bmatrix}
0&\mathcal T(\phi,\theta)&0&0\\
0&A_\omega(\boldsymbol{\omega})&0&0\\
G(\phi,\theta)&0&-[\boldsymbol{\omega}]_\times&0\\
0&0&\mathcal R(\phi,\theta,\psi)&0
\end{bmatrix},
\end{align*}
where every zero denotes a $3\times3$ block. The uncontrolled drift in
\eqref{eq6.17}--\eqref{eq6.18} is $A(\bx)\bx+g\be_9$. The prescribed term
$B_2\bu_{\mathrm{f}}(t)$ completes $\bb_{\mathrm{f}}(t)$ in
\eqref{eq6.19}.

Let $C_\tau=\operatorname{diag}(c_4,c_5,c_6)=\mathcal I^{-1}$.  With disturbance
ordering $\bw_{\mathrm p}=\bw_d=[\mathbf f_w^{\top},\boldsymbol{\tau}_w^{\top}]^{\top}$
and physical input ordering $\bu_{\mathrm{p}}=[T,\boldsymbol{\tau}^{\top}]^{\top}$,
with the same ordering for its nominal and feedback increments, the disturbance
and actuator matrices are
\begin{align*}
B_{1d}^{\rm o}&=
\begin{bmatrix}
0&0\\
0&C_\tau\\
\frac{1}{m}I_3&0\\
0&0
\end{bmatrix},&
B_2&=
\begin{bmatrix}
0_{3\times1}&0_{3\times3}\\
0_{3\times1}&C_\tau\\
-\frac{1}{m}\be_3&0_{3\times3}\\
0_{3\times1}&0_{3\times3}
\end{bmatrix},
\end{align*}
Here $E=B_{1d}^{\rm o}$, and its blocks have three columns each.

The quadrotor state-performance weights are collected analogously in
\begin{align*}
\boldsymbol{\sigma}&=(125,10,10,25,50,50,100,200,200,160),\\
\Sigma&=\operatorname{diag}(\boldsymbol{\sigma}).
\end{align*}
Together with the unit penalty on the feedback increment $\bu$ in
\eqref{eq:quad-input-split}, the performance channel is
\begin{align*}
C_x=\begin{bmatrix}0_{10\times2}&\Sigma\end{bmatrix},\ C_1^{\rm o}=\begin{bmatrix}C_x\\0_{4\times12}\end{bmatrix},\ D_{12}=\begin{bmatrix}0_{10\times4}\\I_4\end{bmatrix}.
\end{align*}
The physical measurement selection and synthesis channel are
\begin{align*}
C=C_2&=
\begin{bmatrix}
I_6&0_{6\times3}&0_{6\times3}\\
0_{3\times6}&0_{3\times3}&I_3
\end{bmatrix},&D_{21,n}&=I_9,
\end{align*}
so the measured components are
$[\phi,\theta,\psi,p,q,r,x,y,z]^{\top}$. The unprojected design channel is
$B_{1d}^{\rm o}$, and the projected CARE channels are specified by
\begin{align*}
P_B&=\operatorname{diag}(0,0,0,1,1,1,0,0,1,0,0,0),\notag\\*
P_C&=\operatorname{diag}(1,1,1,1,1,1,0,0,0,1,1,1),\notag\\*
\widetilde B_{1d}&=P_BB_{1d}^{\rm o},\qquad
\widetilde C_1=C_1^{\rm o}P_C.
\end{align*}
\begin{samepage}
For the full twelve-state Kalman estimator, $Q_w$ is the process-disturbance
covariance used in the filter design:
\begin{align*}
&Q_w=\operatorname{diag}(0.5^2I_3,0.1^2I_3),\\
&W_c=B_{1d}^{\rm o}Q_w(B_{1d}^{\rm o})^{\top}+10^{-4}I_{12},\\
&V=\operatorname{diag}(0.01^2I_3,0.05^2I_6),\qquad P_0=I_{12}.
\end{align*}
\end{samepage}
Here $B_{1d}^{\rm o}$ is the unprojected physical disturbance matrix.
The discrete covariance $W_k$ is obtained from $W_c$ by the same
approximation in Section~\ref{sec:kalman-estimation}, using
$h=0.002$~s for nominal tracking and $h=0.001$~s for the
limited-resource experiment.

For both examples, the dimensions in Section~2 are
\begin{align*}
 n&=12,\quad m_d=6,\quad m_n=9,\\
 m_1&=15,\quad m_2=4,\quad p_2=9.
\end{align*}
The performance dimension is $p_1=16$ for the F-16 and $p_1=14$
for the quadrotor. Completing the disturbance channel with zero columns
preserves its Gram matrix and the threshold in \eqref{eq3.19a}.
The full projected channels satisfy (A3):
\begin{equation*}
 D_{12}^{\top}\widetilde C_1=0,\;
 D_{12}^{\top}D_{12}=I_4,\;
 \widetilde B_1D_{21}^{\top}=0,\;
 D_{21}D_{21}^{\top}=I_9.
\end{equation*}

\section*{Acknowledgements}

T. Li was partially supported by the National Natural Science Foundation of China (NSFC) No. 12371377 and the Jiangsu Provincial Scientific Research Center of Applied Mathematics under Grant No. BK20233002. We thank Tianhe-2 and the Big Data Computing Center in Southeast University, China, for providing access to their computing resources.

The authors used AI for language editing. Figure~\ref{fig6.1} was generated using AI.

\section*{Data availability}

Data will be made available on request.

\bibliographystyle{unsrtnat}
\bibliography{strings,research_papers}
\end{document}

%% file: generated/paper_numbers.tex
\newcommand{\QuadrotorPositionRMSE}{0.0482}
\newcommand{\QuadrotorMaxTrajectoryDifference}{1.42\times 10^{-11}}

%% file: generated/paper_feedforward_numbers.tex
\newcommand{\QuadFeedforwardRMSE}{254.4354}

%% file: generated/paper_timing_table.tex
\begin{tabular}{lrrrrrr}
\toprule
& \multicolumn{2}{c}{\texttt{icare}} & \multicolumn{2}{c}{SDA} & \multicolumn{2}{c}{Newton} \\
\cmidrule(lr){2-3}\cmidrule(lr){4-5}\cmidrule(lr){6-7}
Model & Median (ms) & Speedup & Median (ms) & Speedup & Median (ms) & Speedup \\
\midrule
F-16 Aircraft & 0.865 & $1.0\times$ & 0.205 & $4.2\times$ & 0.171 & $5.0\times$ \\
Quadrotor & 1.117 & $1.0\times$ & 0.237 & $4.7\times$ & 0.282 & $4.0\times$ \\
\bottomrule
\end{tabular}

%% file: research_papers.bib
@String { Elsevier          = {Elsevier Science Publishers} }

@Article{chfl:2005,
  Title                    = {A structure-preserving doubling algorithm for continuous-time algebraic {R}iccati equations},
  Author                   = {E. K.-W. Chu and H.-Y. Fan and W.-W. Lin},
  Journal                  = j-LAA,
  Year                     = {2005},
  Pages                    = {55-80},
  Volume                   = {396},
  Doi                      = {10.1016/j.laa.2004.10.010}
}

@article{hukl:2026,
  author  = {Tsung-Ming Huang and Yueh-Cheng Kuo and Wen-Wei Lin and Chin-Tien Wu},
  title   = {A Mixed-Type {SDRE} Model with a {Kalman} Filter for Prescribed Impact Angle Guidance},
  journal = {Aerospace Science and Technology},
  volume  = {176},
  pages   = {111953},
  year    = {2026},
  doi     = {10.1016/j.ast.2026.111953}
}

@Article{huli:2009,
  author  = {Huang, T.-M. and Lin, W.-W.},
  journal = j-LAA,
  title   = {Structured doubling algorithms for weakly stabilizing {H}ermitian solutions of algebraic {R}iccati equations},
  year    = {2009},
  number  = {5--6},
  pages   = {1452-1478},
  volume  = {430},
  doi     = {10.1016/j.laa.2007.08.043},
}

@Article{hull:2017,
  author  = {Tsung-Ming Huang and R.-C. Li and W.-W. Lin and L.-Z. Lu},
  journal = j-JMS,
  title   = {Optimal Parameters for Doubling Algorithms},
  year    = {2017},
  number  = {4},
  pages   = {339-357},
  volume  = {50},
}

@Article{kimu:1989,
  Title                    = {Doubling algorithm for continuous-time algebraic {R}iccati equation},
  Author                   = {M. Kimura},
  Journal                  = j-IJSS,
  Year                     = {1989},
  Number                   = {2},
  Pages                    = {191-202},
  Volume                   = {20},
  Doi                      = {10.1080/00207728908910119}
}

@Article{laub:1979,
  Title                    = {A {S}chur method for solving algebraic {R}iccati equations},
  Author                   = {A. J. Laub},
  Journal                  = j-IEEE-AC,
  Year                     = {1979},
  Number                   = {6},
  Pages                    = {913-921},
  Volume                   = {24},
  Doi                      = {10.1109/TAC.1979.1102178}
}

@Article{lixu:2006,
  Title                    = {Convergence analysis of structure-preserving doubling algorithms for {R}iccati-type matrix equations},
  Author                   = {Lin, W.-W. and Xu, S.-F.},
  Journal                  = j-SIMAX,
  Year                     = {2006},
  Number                   = {1},
  Pages                    = {26-39},
  Volume                   = {28}
}

@Article{valo:1984,
  Title                    = {A symplectic method for approximating all the eigenvalues of a {H}amiltonian matrix},
  Author                   = {C. F. {Van Loan}},
  Journal                  = j-LAA,
  Year                     = {1984},
  Pages                    = {233-251},
  Volume                   = {61}
}

@Book{hull:2018,
  author    = {Huang, T.-M. and Li, R.-C. and Lin, W.-W.},
  publisher = {Society for Industrial and Applied Mathematics},
  title     = {Structure-Preserving Doubling Algorithms for Nonlinear Matrix Equations},
  year      = {2018},
  address   = {Philadelphia, PA},
}

@Article{arla:1984,
  author  = {Arnold III, W. F. and Laub, A. J.},
  journal = j-Proc-IEEE,
  title   = {Generalized eigenproblem algorithms and software for algebraic {R}iccati equations},
  year    = {1984},
  number  = {12},
  pages   = {1746-1754},
  volume  = {72},
  doi     = {10.1109/PROC.1984.13083},
}

@Article{Cime:2012,
  author    = {{\c{C}}imen, T.},
  journal   = j-JGCD,
  title     = {Survey of state-dependent {R}iccati equation in nonlinear optimal feedback control synthesis},
  year      = {2012},
  number    = {4},
  pages     = {1025--1047},
  volume    = {35},
  doi       = {10.2514/1.55821},
}

@Article{Cime:2010,
  author    = {{\c{C}}imen, T.},
  journal   = j-ANN-REV-CONT,
  title     = {Systematic and effective design of nonlinear feedback controllers via the state-dependent {R}iccati equation ({SDRE}) method},
  year      = {2010},
  number    = {1},
  pages     = {32--51},
  volume    = {34},
  doi       = {10.1016/j.arcontrol.2010.03.001},
  publisher = {Elsevier},
}

@Article{lili:2021,
  author    = {L.-G. Lin and W.-W. Lin},
  journal   = j-IEEE-AES,
  title     = {Computationally Efficient {SDRE} Control Design for 3-{DOF} Helicopter Benchmark System},
  year      = {2021},
  number    = {5},
  pages     = {3320--3336},
  volume    = {57},
  doi       = {10.1109/TAES.2021.3074211},
}

@Article{lilc:2018,
  author    = {Lin, L.-G. and Liang, Y.-W. and Cheng, L.-J.},
  journal   = j-SICOPT,
  title     = {Control for a Class of Second-Order Systems via a State-Dependent {R}iccati Equation Approach},
  year      = {2018},
  number    = {1},
  pages     = {1--18},
  volume    = {56},
  eprint    = {https://doi.org/10.1137/16M1073820},
}

@Article{livl:2015,
  author    = {Lin, L.-G. and Vandewalle, J. and Liang, Y.-W.},
  journal   = j-AUTO,
  title     = {Analytical representation of the state-dependent coefficients in the {SDRE/SDDRE} scheme for multivariable systems},
  year      = {2015},
  pages     = {106--111},
  volume    = {59},
}

@Article{cpws:2022,
  author  = {Chodnicki, M. and Pietruszewski, P. and Weso\l{}owski, M. and St\k{e}pie\'{n}, S.},
  journal = j-ACS,
  title   = {Finite-time {SDRE} control of {F}16 aircraft dynamics},
  year    = {2022},
  number  = {3},
  pages   = {557-576},
  volume  = {32},
  doi     = {10.24425/acs.2022.142848},
}

@Misc{isrlab:F16Model,
  author       = {{Intelligent Systems Research Laboratory}},
  title        = {{F16Model.jl}: Nonlinear model of {F-16} flight dynamics},
  year         = {n.d.},
  howpublished = {Software and aerodynamic data, GitHub},
  url          = {https://github.com/isrlab/F16Model.jl/tree/8b719078922b8fa786554c2446ba749c4ca7a447},
  note         = {Pinned revision 8b719078922b; accessed 11 September 2026},
}

@InProceedings{clou:1997,
  author    = {Cloutier, J. R.},
  booktitle = {Proceedings of the 1997 American Control Conference (Cat. No.97CH36041)},
  title     = {State-dependent {R}iccati equation techniques: an overview},
  year      = {1997},
  pages     = {932-936},
  volume    = {2},
  doi       = {10.1109/ACC.1997.609663},
}

@Article{smit:1968,
  author  = {R. A. Smith},
  journal = j-SIAMA,
  title   = {Matrix Equation ${XA + BX = C}$},
  year    = {1968},
  number  = {1},
  pages   = {198-201},
  volume  = {16},
}

@Article{hklw:2024,
  author   = {T.-M. Huang and Y.-C. Kuo and W.-W. Lin and C.-T. Wu},
  journal  = {Aerosp. Sci. Technol.},
  title    = {An optimal parameterized {N}ewton-type structure-preserving doubling algorithm for impact angle guidance-based {3D} pursuer/target interception engagement},
  year     = {2024},
  issn     = {1270-9638},
  pages    = {109674},
  volume   = {155},
  doi      = {https://doi.org/10.1016/j.ast.2024.109674},
  url      = {https://www.sciencedirect.com/science/article/pii/S1270963824008034},
}

@InProceedings{ChHu:2022,
  author    = {Cheng, Sheng-Wen and Hung, Hsin-Ai},
  booktitle = {2022 International Automatic Control Conference (CACS)},
  title     = {Robust State-Feedback ${H}_{\infty}$ Control of Quadrotor},
  year      = {2022},
  pages     = {1-7},
  doi       = {10.1109/CACS55319.2022.9969787},
}

@Article{neao:2022,
  author  = {Nekoo, S. R. and Acosta, J. \'{A}. and Ollero, A},
  journal = {Robotica},
  title   = {Quaternion-based state-dependent differential {R}iccati equation for quadrotor drones: {R}egulation control problem in aerobatic flight},
  year    = {2022},
  number  = {9},
  pages   = {3120-3135},
  volume  = {40},
  doi     = {10.1017/S0263574722000091},
}

@Article{Hewer:1993,
author = {Hewer, Gary},
title = {Existence Theorems for Positive Semidefinite and Sign Indefinite Stabilizing Solutions of {$H_\infty$} {Riccati} Equations},
journal = {SIAM Journal on Control and Optimization},
volume = {31},
number = {1},
pages = {16-29},
year = {1993},
doi = {10.1137/0331002},
}

@Article{dgkf:1989,
   author={Doyle, J.C. and Glover, K. and Khargonekar, P.P. and Francis, B.A.},
  journal={IEEE Transactions on Automatic Control}, 
  title={State-space solutions to standard {$H_2$} and {$H_{\infty}$} control problems}, 
   year={1989},
  volume={34},
  number={8},
  pages={831-847},
  doi={10.1109/9.29425}}

@InProceedings{eraa:2001,
  author    = {Erdem, Evrin B. and Andrew G. Alleyne},
  booktitle = {Proceedings of the 40th IEEE conference on decision and control (Cat. No. 01CH37228)},
  title     = {Experimental real-time {SDRE} control of an underactuated robot},
  year      = {2001},
  pages     = {2986-2991},
  volume    = {3},
  doi       = {10.1109/CDC.2001.980731},
}

@InProceedings{fgim:2015,
 author={Franzini, Giovanni and Innocenti, Mario},
  booktitle={2015 54th IEEE Conference on Decision and Control (CDC)}, 
  title={Nonlinear {H}-infinity control of relative motion in space via the state-dependent {Riccati} equations}, 
  year={2015},
  volume={},
  number={},
  pages={3409-3414},
  doi={10.1109/CDC.2015.7402733}}

@Article{beit:2012,
  author  = {Beikzadeh, Hossein and Taghirad, Hamid D.},
  journal = {ISA Transactions},
  title   = {Robust {SDRE} filter design for nonlinear uncertain systems with an ${H}_{\infty}$ performance criterion},
  year    = {2012},
  number  = {1},
  pages   = {146--152},
  volume  = {51},
  doi     = {10.1016/j.isatra.2011.09.003},
}

@Article{makc:2012,
  author  = {Mahony, Robert and Kumar, Vijay and Corke, Peter},
  journal = {IEEE Robotics \& Automation Magazine},
  title   = {Multirotor aerial vehicles: modeling, estimation, and control of quadrotor},
  year    = {2012},
  number  = {3},
  pages   = {20--32},
  volume  = {19},
  doi     = {10.1109/MRA.2012.2206474},
}

@Article{klei:1968,
  author  = {Kleinman, David L.},
  journal = j-IEEE-AC,
  title   = {On an iterative technique for {R}iccati equation computations},
  year    = {1968},
  number  = {1},
  pages   = {114--115},
  volume  = {13},
  doi     = {10.1109/TAC.1968.1098829},
}

@Article{feit:2009,
  author  = {Feitzinger, F. and Hylla, T. and Sachs, E. W.},
  journal = {SIAM Journal on Matrix Analysis and Applications},
  title   = {Inexact {K}leinman--{N}ewton method for {R}iccati equations},
  year    = {2009},
  number  = {2},
  pages   = {272--288},
  volume  = {31},
  doi     = {10.1137/070700978},
}

@Book{simo:2006,
  author    = {Dan Simon},
  publisher = {John Wiley \& Sons},
  title     = {Optimal State Estimation: {K}alman, {$H_{\infty}$}, and Nonlinear Approaches},
  year      = {2006},
}

@Article{lxy:2001,
  author  = {Wen-Wei Lin and Quan-Fu Xu and Fang-Bo Yeh},
  journal = {Annals of Operations Research},
  title   = {Invariant Subspace Approach to Linear {$H_{\infty}$}-Control via Measurement Feedback},
  year    = {2001},
  pages   = {379–388},
  volume  = {103},
  doi     = {10.1023/A:1012949827197},
}

@Article{gd:1988,
author = {Keith Glover and John C. Doyle},
title = {State-space formulae for all stabilizing controllers that satisfy an {$H_\infty$}-norm bound and relations to risk sensitivity},
journal = {Systems \& Control Letters},
volume = {11},
number = {3},
pages = {167-172},
year = {1988},
issn = {0167-6911},
doi = {https://doi.org/10.1016/0167-6911(88)90055-2},
url = {https://www.sciencedirect.com/science/article/pii/0167691188900552}
}

@Article{bowa:2007,
  author  = {Bogdanov, Alexander and Wan, Eric A.},
  journal = {Journal of Guidance, Control, and Dynamics},
  title   = {State-Dependent {R}iccati Equation Control for Small Autonomous Helicopters},
  year    = {2007},
  number  = {1},
  pages   = {47-60},
  volume  = {30},
  doi     = {10.2514/1.21910},
}

@Article{neol:2025,
  author  = {Nekoo, Saeed Rafee and Ollero, Anibal},
  journal = {Drone Systems and Applications},
  title   = {Experimental implementation of state-dependent {R}iccati equation control on quadrotors},
  year    = {2025},
  pages   = {1-16},
  volume  = {13},
  doi     = {10.1139/dsa-2024-0062},
}

@Article{lich:1993,
  author  = {Li, X. P. and Chang, B. C.},
  title   = {On the convexity of {$H_\infty$} {Riccati} solutions and its applications},
  journal = {IEEE Transactions on Automatic Control},
  year    = {1993},
  volume  = {38},
  number  = {6},
  pages   = {963--966},
  doi     = {10.1109/9.222311},
}


%% file: strings.bib
@PREAMBLE{"\def\noopsort#1{}"
# "\def\v#1{{\accent20 #1}} \let\^^_=\v"
# "\def\hbk{hardback}"
# "\def\pbk{paperback}"
}

@String{j-ANN-REV-CONT              = "Annu. Rev. Control"}

@String{j-ACS                  = "Arch. Control Sci."}

@String{j-AUTO                  = "Automatica"}

@String{j-IEEE-AC               = "IEEE Trans. Automat. Control"}

@String{j-IEEE-AES               = "IEEE Trans. Aerosp. Electron. Syst."}

@String{j-IJSS                  = "Int. J. Syst. Sci."}

@String{j-JGCD                  = "J. Guid. Control Dyn."}

@String{j-JMS                   = "J. Math. Study"}

@String{j-LAA                   = "Linear Algebra Appl."}

@string{j-PROC-IEEE             = "Proc. IEEE"}

@String{j-SIAMA                 = "SIAM J. Appl. Math."}

@String{j-SICOPT                = "SIAM J. Control Optim."}

@String{j-SIMAX                 = "SIAM J. Matrix Anal. Appl."}
